\documentclass[11pt,reqno]{amsart}
\usepackage[breaklinks]{hyperref}
\usepackage{url}
\usepackage{amssymb}
\newcommand{\df}{\dfrac}
\newcommand{\tf}{\tfrac}

 \renewcommand{\a}{\alpha}
\renewcommand{\b}{\beta}

\newcommand{\g}{\gamma}
\newcommand{\G}{\Gamma}
\renewcommand{\l}{\lambda}

\renewcommand{\(}{\left\(}
\renewcommand{\)}{\right\)}
\renewcommand{\[}{\left\[}
\renewcommand{\]}{\right\]}

\numberwithin{equation}{section}
 \theoremstyle{plain}
\newtheorem{theorem}{Theorem}[section]
\newtheorem{lemma}[theorem]{Lemma}
\newtheorem{corollary}[theorem]{Corollary}

\newtheorem{definition}[theorem]{Definition}

   \makeatletter
\def\proof{\@ifnextchar[{\@oproof}{\@nproof}}
\def\@oproof[#1][#2]{\trivlist\item[\hskip\labelsep\textit{#2 Proof of\
#1.}~]\ignorespaces}
\def\@nproof{\trivlist\item[\hskip\labelsep\textit{Proof.}~]\ignorespaces}

\makeatother

\begin{document}
\title[A generalized modified Bessel function]{A generalized modified Bessel function and a higher level analogue of the theta transformation formula}
\author{Atul Dixit, Aashita Kesarwani, and Victor H. Moll}
\address{Department of Mathematics, Indian Institute of Technology, Gandhinagar, Palaj, Gandhinagar 382355, Gujarat, India}\email{adixit@iitgn.ac.in}
\address{Department of Mathematics, Tulane University , New Orleans 70118, USA} \email{akesarwa@tulane.edu}
\address{Department of Mathematics, Tulane University , New Orleans 70118, USA} \email{vhm@tulane.edu}
\dedicatory{\emph{with an appendix by NICO M. TEMME}}
\address{IAA, 1825 BD 25, Alkmaar. Former address: Centrum Wiskunde \& Informatica, 1098 XG Amsterdam, The Netherlands}\email{nicot@cwi.nl}
\begin{abstract}
A new generalization of the modified Bessel function of the second kind $K_{z}(x)$ is studied. Elegant series and integral representations, a differential-difference equation and asymptotic expansions are obtained for it thereby anticipating a rich theory that it may possess. The motivation behind introducing this generalization is to have a function which gives a new pair of functions reciprocal in the Koshliakov kernel $\cos \left( {{\pi z}} \right){M_{2z}}(4\sqrt {x} ) - \sin \left( {{\pi z}} \right){J_{2z}}(4\sqrt {x} )$ and which subsumes the self-reciprocal pair involving $K_{z}(x)$. Its application towards finding modular-type transformations of the form $F(z, w, \a)=F(z,iw,\b)$, where $\a\b=1$, is given. As an example, we obtain a beautiful generalization of a famous formula of Ramanujan and Guinand equivalent to the functional equation of a non-holomorphic Eisenstein series on $SL_{2}(\mathbb{Z})$. 
This generalization can be considered as a higher level analogue of the general theta transformation formula. We then use it to evaluate an integral involving the Riemann $\Xi$-function and consisting of a sum of products of two confluent hypergeometric functions.
\end{abstract}
\maketitle

\section{Introduction}\label{intro}
Bessel functions are among the most important special functions of mathematics. Just in mathematics, they encompass differential equations, integral transforms, number theory (especially analytic number theory), Maass forms and mock modular forms, to name a few. 
The Bessel functions of the first and second kinds of order $\nu$, namely $J_{\nu}(\lambda)$ and $Y_{\nu}(\lambda)$, are defined by  \cite[p.~40, 64]{watson-1944a}
\begin{align}\label{sumbesselj}
	J_{\nu}(\lambda):=\sum_{m=0}^{\infty}\frac{(-1)^m(\lambda/2)^{2m+\nu}}{m!\Gamma(m+1+\nu)}, \quad |\l|<\infty,
	\end{align}
	and
	\begin{align*}
Y_{\nu}(\l)=\frac{J_{\nu}(\l)\cos(\pi \nu)-J_{-\nu}(\l)}{\sin{\pi \nu}}\label{yj}
\end{align*}
respectively. The modified Bessel functions of the first and second kinds of  order $\nu$ are defined by \cite[p.~77]{watson-1944a}
\begin{equation}\label{besseli}
I_{\nu}(\lambda)=
\begin{cases}
e^{-\frac{1}{2}\pi\nu i}J_{\nu}(e^{\frac{1}{2}\pi i}\lambda), & \text{if $-\pi<$ arg $\lambda\leq\frac{\pi}{2}$,}\\
e^{\frac{3}{2}\pi\nu i}J_{\nu}(e^{-\frac{3}{2}\pi i}\lambda), & \text{if $\frac{\pi}{2}<$ arg $\lambda\leq \pi$,}
\end{cases}
\end{equation}
and \cite[p.~78]{watson-1944a}
\begin{equation*}
K_{\nu}(\lambda):=\frac{\pi}{2}\frac{I_{-\nu}(\lambda)-I_{\nu}(\lambda)}{\sin\nu\pi}.
\end{equation*}
Watson's treatise \cite{watson-1944a} is a monumental work on Bessel functions and to this day remains a standard reference.

Several generalizations of the Bessel as well as of the modified Bessel functions have been studied over the years. For these generalizations, we refer the reader to a recent article \cite{masparpog} and the references therein. One of the goals of this paper is to study an interesting new generalization of the modified Bessel function of the second kind. For $z, w \in\mathbb{C}$, $x\in\mathbb{C}\backslash\{x\in\mathbb{R}: x\leq 0\}$ , 
and Re$(s)>\pm$ Re$(z)$, we define this generalized modified Bessel function by an inverse Mellin transform, namely,
\begin{equation}\label{kzw}
K_{z,w}(x) := 
\frac{1}{2\pi i} \int_{(c)}\Gamma\bigg(\frac{s-z}{2}\bigg) \Gamma\bigg(\frac{s+z}{2}\bigg) \,_1F_1\bigg(\frac{s-z}{2};\frac{1}{2};\frac{-w^2}{4}\bigg)  
\,_1F_1\bigg(\frac{s+z}{2};\frac{1}{2};\frac{-w^2}{4}\bigg) 
 2^{s-2}x^{-s}ds,
\end{equation}
where $\G(s)$ denotes the gamma function and ${}_1F_{1}(a;c;w)$ is the confluent hypergeometric function defined by \cite[p.~188]{aar}
\begin{equation*}
{}_1F_{1}(a;c;z)=\sum_{n=0}^{\infty}\frac{(a)_{n}z^{n}}{(c)_{n}n!},
\end{equation*}
with $(a)_{n}$ being the rising factorial $(a)_{n}:=a(a+1)\cdots (a+n-1)=\Gamma(a+n)/\Gamma(a)$ for $a\in\mathbb{C}$. Here, and throughout the sequel, $\int_{(c)}$ denotes the line integral $\int_{c-i\infty}^{c+i\infty}$.

From \eqref{kzw}, one property of $K_{z,w}(x)$ follows immediately, namely, that it is an even function in both the variables $z$ and $w$. When $w=0$, $K_{z,w}(x)$ reduces to the usual modified Bessel function $K_{z}(x)$ owing to the fact that \cite[p.~115, formula 11.1]{ober} for $c=$ Re $s>\pm$ Re $\nu$,
\begin{equation}\label{bessmel}
\frac{1}{2\pi i}\int_{(c)}2^{s-2}\Gamma\left(\frac{s}{2}-\frac{\nu}{2}\right)\Gamma\left(\frac{s}{2}+\frac{\nu}{2}\right)x^{-s}\, ds=K_{\nu}(x).
\end{equation}
The motivation behind the introduction of the generalized modified Bessel function in \eqref{kzw} is now explained. The transformation formula for the Jacobian theta function can be put in an equivalent symmetric form \cite{dixthet}
\begin{align}\label{ttum}
\sqrt{\alpha}\bigg(\frac{1}{2\alpha}-\sum_{n=1}^{\infty}e^{-\pi\alpha^2n^2}\bigg)=\sqrt{\beta}\bigg(\frac{1}{2\beta}-\sum_{n=1}^{\infty}e^{-\pi\beta^2n^2}\bigg),
\end{align}
where $\a$ and $\b$ are complex numbers with Re$(\a^2)>0$, Re$(\b^2)>0$, and satisfying $\a\b=1$. Either side of this formula is equal to an integral involving the Riemann $\Xi$-function, that is \cite{dixthet}, to
\begin{equation}\label{titchint1}
\frac{2}{\pi}\int_{0}^{\infty}\frac{\Xi(t/2)}{1+t^2}\cos\bigg(\frac{1}{2}t\log\alpha\bigg)\, dt,
\end{equation}
where
\begin{equation*}
\Xi(t):=\xi(\tfrac{1}{2}+it),
\end{equation*}
$\xi(s)$ being the Riemann $\xi$-function
\begin{equation*}
\xi(s):=\tfrac{1}{2}s(s-1)\pi^{-\frac{1}{2}s}\Gamma(\tfrac{s}{2})\zeta(s),
\end{equation*}
with $\zeta(s)$ the Riemann zeta function. The equality between this integral involving $\Xi$-function and either side of the theta transformation formula was employed by Hardy \cite{ghhcr} in his famous proof of the infinitude of the zeros of $\zeta(s)$ on the critical line Re$(s)=1/2$. 

Inherent in the equality of the two expressions in \eqref{ttum} with that in \eqref{titchint1} is Laplace's integral evaluation
\begin{equation}\label{e}
e^{-\a^2x^2}=\frac{2}{\a\sqrt{\pi}}\int_{0}^{\infty}e^{-u^2/\a^2}\cos(2ux)\, du,
\end{equation}
implying that, up to a constant, the function $e^{-x^2}$ is self-reciprocal in the Fourier cosine transform. The latter fact can then be used to obtain not only the theta transformation but also formulas of Hardy \cite{ghh} and Ferrar \cite{ferrar}.

However, Laplace's result can be generalized to \cite[p.~527, Formula \textbf{4.133.2}]{grn}
\begin{equation}\label{egen}
e^{-\a^2x^2}\cos(wx)=\frac{2e^{-w^2/(4\a^2)}}{\a\sqrt{\pi}}\int_{0}^{\infty}e^{-u^2/\a^2}\cosh(wu/\a^2)\cos(2ux)\, du,
\end{equation}
where $w\in\mathbb{C}$. Note that if we now replace $x$ by $u$ in \cite[p.~527, Formula \textbf{4.133.2}]{grn}, and then let $\b=2ix, i=\sqrt{-1}$, $a=w$ and $\g=1/(4\a^2)$, then we obtain
\begin{equation}\label{egen1}
e^{-x^2/\a^2}\cosh(wx/(\a^2))=\frac{2\a e^{w^2/(4\a^2)}}{\sqrt{\pi}}\int_{0}^{\infty}e^{-\a^2u^2}\cos(wu)\cos(2ux)\, du.
\end{equation}
This implies that up to a constant $e^{-x^2}\cos(wx)$ and $e^{-x^2}\cosh(wx)$ are reciprocal functions in the Fourier cosine transform. This fact is inherent in obtaining the generalization of the theta transformation, that is, for $\a\b=1$ and $w\in\mathbb{C}$,
\begin{align}\label{eqsym0}
\sqrt{\alpha}\bigg(\frac{e^{-\frac{w^2}{8}}}{2\alpha}-e^{\frac{w^2}{8}}\sum_{n=1}^{\infty}e^{-\pi\alpha^2n^2}\cos(\sqrt{\pi}\alpha nw)\bigg)
&=\sqrt{\beta}\bigg(\frac{e^{\frac{w^2}{8}}}{2\beta}-e^{-\frac{w^2}{8}}\sum_{n=1}^{\infty}e^{-\pi\beta^2n^2}\cosh(\sqrt{\pi}\beta nw)\bigg),
\end{align}
as well as the generalizations of the formulas of Ferrar and Hardy derived in \cite{dixthet}, where the integrals involving the Riemann $\Xi$-function associated to them now take the form
\begin{equation}\label{sp0}
\int_{0}^{\infty}f\left(\frac{t}{2}\right)\Xi\left(\frac{t}{2}\right)\nabla\left(\alpha,w,\frac{1+it}{2}\right)\, dt,
\end{equation}
where $f(t)$ is of the form $f(t)=\phi(it)\phi(-it)$ with $\phi$ being analytic in $t$ as a function of a real variable, and 
\begin{align}\label{nabla}
\nabla(x,w,s)&:=\rho(x,w,s)+\rho(x,w,1-s),\nonumber\\
\rho(x,w,s)&:=x^{\frac{1}{2}-s}e^{-\frac{w^2}{8}}{}_1F_{1}\left(\frac{1-s}{2};\frac{1}{2};\frac{w^2}{4}\right).
\end{align}
Note that $\nabla\left(\alpha,0,\frac{1+it}{2}\right)=\alpha^{-\frac{it}{2}}+\alpha^{\frac{it}{2}}=2\cos\left(\frac{1}{2}t\log\alpha\right)$ so that when $f(t)=\frac{1}{\pi(1+4t^2)}$, the integral in \eqref{sp0}, which is equal to each side of \eqref{eqsym0} (as shown in \cite[Theorem 1.2]{dixthet}), reduces to that in \eqref{titchint1}. General theorems in this regard which work for any pair of reciprocal functions in Fourier cosine transform, such as the pair in \eqref{egen} and \eqref{egen1}, were obtained in \cite[Theorems 1.2, 1.3 and 1.4]{drrz01}.

We now note that while $e^{-x^2}$ is the inverse Mellin transform of essentially the gamma function $\G(s/2)$, $K_0(x)$ is the inverse Mellin transform of essentially the square of the gamma function, that is, $\G^2(s/2)$. Further, $K_z(x)$ is essentially the inverse Mellin transform of $\G\left(\frac{s-z}{2}\right)\G\left(\frac{s+z}{2}\right)$, as can be seen from \eqref{bessmel}. While we have seen above that the natural kernel when working with transformations involving $e^{-x^2}$ is the cosine function in the sense that it renders $e^{-x^2}$ as self-reciprocal when integrated against the cosine (see \eqref{e}), the natural kernel while working with $K_{z}(x)$ is the Koshliakov kernel given by $\cos \left( {{\pi z}} \right){M_{2z}}(4\sqrt {x} ) - \sin \left( {{\pi z}} \right){J_{2z}}(4\sqrt {x} )$, with ${M_z }(x) := \frac{2}{\pi }{K_z }(x) - {Y_z }(x)$, for, Koshliakov \cite{kosh1938} showed that for $-\tfrac{1}{2}<z<\tfrac{1}{2}$ \footnote{It is easy to see that this identity actually holds for  $ - \tfrac{1}{2}<$ Re$(z)<\tfrac{1}{2}$.},
\begin{equation}\label{koshlyakov-1}
2\int_{0}^{\infty} K_{z}(2t) \left( \cos(\pi z) M_{2z}(4 \sqrt{xt}) -
\sin(\pi z) J_{2z}(4 \sqrt{xt}) \right)\, dt = K_{z}(2x).
\end{equation}
Now a natural question arises - if \eqref{e} can be generalized to \eqref{egen}, does there exist a one-variable generalization of \eqref{koshlyakov-1} which would then give us a pair of functions reciprocal in the Koshliakov kernel? We answer this question in the affirmative. The right function that plays in this case the role played by $e^{-x^2}\cos(wx)$ (in the case of Fourier cosine transform) is the generalized modified Bessel function $K_{z,w}(x)$ defined in \eqref{kzw}. Our first result given below proves this claim.
\begin{theorem}\label{recpairkosh}
Let $-\frac{1}{2}<$ \textup{Re}$(z)<\frac{1}{2}$. Let $w\in\mathbb{C}$ and $x>0$. Let $\a$ and $\b$ be two positive numbers such that $\a\b=1$. The functions $e^{-\frac{w^2}{2}} K_{z,iw}(2\a x)$ and $\b \,K_{z,w}(2 \b x)$ form a pair of reciprocal functions in the Koshliakov kernel, that is,
\begin{align}\label{recpairkosheqn}
&e^{-\frac{w^2}{2}} K_{z,iw}(2\a x)=2\int_{0}^{\infty} \b \,K_{z,w}(2 \b t) \left( \cos(\pi z) M_{2z}(4 \sqrt{xt}) -
\sin(\pi z) J_{2z}(4 \sqrt{xt}) \right)\, dt,\nonumber\\
&\b \,K_{z,w}(2 \b x)=2\int_{0}^{\infty}e^{-\frac{w^2}{2}} K_{z,iw}(2\a t)\left( \cos(\pi z) M_{2z}(4 \sqrt{xt}) -
\sin(\pi z) J_{2z}(4 \sqrt{xt}) \right)\, dt.
\end{align}
\end{theorem}
The required asymptotics of $K_{z,w}(x)$ which render the convergence of the above integrals are obtained in Theorems \ref{thm:Kexpansion} and \ref{kzwsmall} below.

It should be mentioned here that there are very few functions that lead to an exact evaluation when integrated against the Koshliakov kernel over the positive real line. This has been recorded by Dixon and Ferrar \cite[p.~161]{dixfer3} for the special case $z=0$ of the Koshliakov kernel. Koshliakov \cite{kosh1938} generalized the results of Dixon and Ferrar from \cite{dixfer3} for real $z$. In view of this, it is nice to be able to add the new reciprocal pair $(e^{-w^2/2} K_{z,iw}(2\a x), \b \,K_{z,w}(2 \b x))$ to the list.

An application of Theorem \ref{recpairkosh} is in evaluating certain integrals involving the Riemann $\Xi$-function, and containing three parameters, namely, $\a, z$ and $w$. To see this, we need to use some terminology from \cite{koshkernel} but with the accommodation of the extra parameter $w$ that we have in the associated functions.

\begin{definition}
Let $0<\omega\leq \pi$ and $\eta>0$. For fixed $z$ and $w$, if $u(s, z, w)$ is such that
\begin{enumerate}
\item[(i)] $u(s, z, w)$ is an analytic function of $s=re^{i\theta}$ regular in the angle defined by $r>0$, $|\theta|<\omega$,
\item[(ii)] $u(s, z, w)$ satisfies the bounds
\begin{equation}\label{growth}
u(s, z, w)=
			\begin{cases}
			O_{z, w}(|s|^{-\delta}) & \mbox{ if } |s| \le 1,\\
			{O_{z, w}(|s|^{-\eta-1-|\textup{Re}(z)|})} & \mbox{ if } |s| > 1,
			\end{cases}
\end{equation}
\end{enumerate}
for every positive $\delta$ and uniformly in any angle $|\theta|<\omega$, then we say that $u$ belongs to the class $\Diamond_{\eta, \omega}$ and write $u(s, z, w)\in \Diamond_{\eta,\omega}$.
\end{definition}

Let the functions $\varphi$ and $\psi$ be related by
\begin{align} \label{recip1}
\varphi (x, z, w) &= 2\int_0^\infty  {\psi (t, z, w)\left( {\cos \left( {\pi z} \right){M_{2z}}(4\sqrt {tx} ) - \sin \left( \pi z \right){J_{2z}}(4\sqrt {tx} )} \right)\, dt},\nonumber\\
\psi (x, z, w) &= 2\int_0^\infty  {\varphi (t, z, w)\left( {\cos \left( {\pi z} \right){M_{2z}}(4\sqrt {tx} ) - \sin \left( \pi z \right){J_{2z}}(4\sqrt {tx} )} \right)\, dt} .
\end{align}

Define the normalized Mellin transforms $Z_1(s, z, w)$ and $Z_2(s, z, w)$ of the functions $\varphi(x, z, w)$ and $\psi(x, z, w)$ by
\begin{align*}
\Gamma \left( {\frac{s-z}{2}} \right)\Gamma \left( {\frac{s+z}{2}} \right){Z_1}(s, z, w) &= \int_0^\infty  {x^{s - 1}}{\varphi (x, z, w)\, dx},\\
 \Gamma \left( {\frac{s-z}{2}} \right)\Gamma \left( {\frac{s+z}{2}} \right){Z_2}(s, z, w) &= \int_0^\infty  {x^{s - 1}}{\psi (x, z, w)\, dx},
\end{align*}
where each equation is valid in a specific vertical strip in the complex $s$-plane. Set 
\begin{equation}\label{add}
Z(s, z, w):= Z_1(s, z, w) + Z_2(s, z, w) \quad \textnormal{and} \quad \Theta(x, z, w) := \varphi(x, z, w) + \psi(x, z, w),
\end{equation}
so that
\begin{align*}
\Gamma \left( {\frac{s-z}{2}}\right)\Gamma \left( {\frac{s+z}{2}}\right)Z(s, z, w) = \int_0^\infty  {{x^{s - 1}}\Theta (x, z, w)\, dx} 
\end{align*}
for values of $s$ in the intersection of the two vertical strips.\\

In \cite[Theorem 1.2]{koshkernel}, the following result was proved:

\textit{ Let $\eta>1/4$ and $0<\omega\leq \pi$. Suppose that $\varphi,\psi \in \Diamond_{\eta,\omega}$, are reciprocal in the Koshliakov kernel, and that $-1/2<$ \textup{Re}$(z)<1/2$. Let $Z(s, z, w)$ and $\Theta(x, z, w)$ be defined in \eqref{add}. Let $\sigma_{-z}(n)=\sum_{d|n}d^{-z}$. Then,
\begin{align}\label{ramguigeneid}
  &\frac{32}{\pi}\int_0^\infty  {\Xi \left( {\frac{{t + iz}}{2}} \right)\Xi \left( {\frac{{t - iz}}{2}} \right)Z\left( {\frac{{1 + it}}{2}}, \frac{z}{2},w \right)\frac{{dt}}{{({t^2}+(z+1)^2)({t^2}+(z-1)^2)}}}  \nonumber \\
  &= \sum\limits_{n = 1}^\infty  {{\sigma _{ - z}}(n){n^{z/2}}\Theta \left(\pi n, \frac{z}{2},w\right)} -R(z, w), 
\end{align}
where
\begin{align*}
{R}(z, w) := {\pi ^{z/2}}\Gamma \left( {\frac{{-z}}{2}} \right)\zeta (-z)Z\left(1 + \frac{z}{2},\frac{z}{2}, w\right) + {\pi ^{-z/2}}\Gamma \left( {\frac{{z}}{2}} \right)\zeta (z)Z\left(1-\frac{z}{2},\frac{z}{2}, w\right).
\end{align*}}

Now let $\nabla_{2}(x, z, w, s)$ be defined by
\begin{equation}\label{nabrho}
\nabla_{2}(x, z, w, s):=\rho(x, z, w, s)+\rho(x, z, w, 1-s),
\end{equation}
where
\begin{equation*}
\rho(x, z, w, s):=x^{\frac{1}{2}-s}{}_1F_{1}\left(\frac{1-s-z}{2};\frac{1}{2};-\frac{w^2}{4}\right){}_1F_{1}\left(\frac{1-s+z}{2};\frac{1}{2};-\frac{w^2}{4}\right).
\end{equation*}
The notation $\nabla_{2}(x, z, w, s)$ reflects similarity with that used in \cite[Eqns. (1.8), (1.9)]{dixthet}\footnote{The role of $z$ there is played by $w$ here (see \eqref{nabla} above), and so $z$ here is a new variable that was not present in those equations.}.

Using the reciprocal pair from Theorem \ref{recpairkosh} in \eqref{ramguigeneid}, we deduce the following result which generalizes Corollary 1.3 of \cite{koshkernel}.
\begin{theorem}\label{xiintgenrgthm}
Let $w\in\mathbb{C}$ and $-1<$ \textup{Re}$(z)<1$. Let $K_{z,w}(x)$ and $\nabla_{2}(x, z, w, s)$ be defined in \eqref{kzw} and \eqref{nabrho} respectively.  If $\a$ and $\b$ are positive integers satisfying $\a\b=1$, then
\begin{align}\label{xiintgenrg}
&\frac{16}{\pi} \int_{0}^{\infty} \Xi\left( \frac{t+iz}{2} \right) \Xi\left( \frac{t-iz}{2} \right)
\frac{\nabla_{2}\left(\a,\tfrac{z}{2},w,\tfrac{1+it}{2}\right)\, dt}{\left(t^2+(z+1)^2\right)\left(t^2+(z-1)^2\right)}\nonumber\\
&=  e^{-\frac{w^2}{4}} \sqrt{\a} \bigg\{ 
4 \sum_{n=1}^{\infty} \sigma_{-z}(n) n^{\frac{z}{2}} e^{-\frac{w^2}{4}} K_{\frac{z}{2},iw}(2n\pi \a)
- \Gamma\left(\frac{z}{2}\right) \zeta(z)  \pi^{-\frac{z}{2}} \a^{\frac{z}{2}-1} \,_1F_1\left(\frac{1-z}{2};\frac{1}{2};\frac{w^2}{4}\right)\nonumber\\
&\quad\quad\quad\quad\quad\quad- \Gamma\left(-\frac{z}{2}\right) \zeta(-z) \pi^{\frac{z}{2}}  \a^{-\frac{z}{2}-1} \,_1F_1\left(\frac{1+z}{2};\frac{1}{2};\frac{w^2}{4} \right)  
\bigg\}.
\end{align}
\end{theorem}
Since 
\begin{equation*}
\nabla_{2}\left(\a,\tfrac{z}{2},0,\tfrac{1+it}{2}\right)=\a^{-it/2}+\a^{it/2}=2\cos\left(\tfrac{1}{2}t\log\a\right),
\end{equation*}
Theorem \ref{xiintgenrgthm} gives Corollary 1.3 from \cite{koshkernel} as a special case.

The Ramanujan-Guinand formula \cite[p.~253]{lnb}, \cite{guinand}, gives, for $ab=\pi^2$,
\begin{multline}\label{mainagain}
\sqrt{a}\sum_{n=1}^{\infty}\sigma_{-z}(n)n^{z/2}K_{z/2}(2na)
-\sqrt{b}\sum_{n=1}^{\infty}\sigma_{-z}(n)n^{z/2}K_{z/2}(2nb)\\
=\df{1}{4}\Gamma\left(\dfrac{z}{2}\right)\zeta(z)\{b^{(1-z)/2}-a^{(1-z)/2}\}
+\df{1}{4}\Gamma\left(-\dfrac{z}{2}\right)\zeta(-z)\{b^{(1+z)/2}-a^{(1+z)/2}\},
\end{multline}
where $\sigma_{z}(n)=\sum_{d|n}d^{z}$. (See \cite{bls} for history and other results derived from it.)

As an intermediate step in the proof of Theorem \ref{xiintgenrgthm}, we obtain the following elegant generalization of the Ramanujan-Guinand formula.
\begin{theorem}\label{genrg}
Let $z, w\in\mathbb{C}$. Let $K_{z,w}(x)$ be defined in \eqref{kzw}. For $a, b>0$ such that $a \,b = \pi^2$,
\begin{align}\label{genrgeqn}
&\sqrt{a} \sum_{n=1}^{\infty} \sigma_{-z}(n) n^{z/2}  e^{-\frac{w^2}{4}} K_{\frac{z}{2},iw}(2  n a) -   \sqrt{b} \sum_{n=1}^{\infty} \sigma_{-z}(n) n^{z/2}  e^{\frac{w^2}{4}}  K_{\frac{z}{2},w}(2  n b) \nonumber\\
&=  \frac{1}{4} \Gamma\left(\frac{z}{2}\right) \zeta(z) \bigg\{ b^{\frac{1-z}{2}} \,_1F_1\left(\frac{1-z}{2};\frac{1}{2};\frac{w^2}{4}\right) - a^{\frac{1-z}{2}} \,_1F_1\left(\frac{1-z}{2};\frac{1}{2};-\frac{w^2}{4}\right)\bigg\} \nonumber\\
&\quad+\frac{1}{4} \Gamma\left(-\frac{z}{2}\right) \zeta(-z)  \bigg\{ b^{\frac{1+z}{2}} \,_1F_1\left(\frac{1+z}{2};\frac{1}{2};\frac{w^2}{4} \right) - a^{\frac{1+z}{2}} \,_1F_1\left(\frac{1+z}{2};\frac{1}{2};-\frac{w^2}{4} \right) \bigg\}.
\end{align}
\end{theorem}
We prove this result in its following equivalent form.
\begin{theorem}\label{genrgeq}
Let $w\in\mathbb{C}$, $z\in\mathbb{C}\backslash\{-1,1\}$. For $\a,\b>0$ such that $\a\b = 1$,
\begin{align}\label{genrgeqeqn}
&\sqrt{\a}  \Bigg( 
4 \sum_{n=1}^{\infty} \sigma_{-z}(n) n^{\frac{z}{2}}  e^{-\frac{w^2}{4}} K_{\frac{z}{2},iw}(2  n \pi\a)
- \Gamma\left(\frac{z}{2}\right) \zeta(z)  \pi^{-\frac{z}{2}} \a^{\frac{z}{2}-1} 
\,_1F_1\left(\frac{1-z}{2};\frac{1}{2};\frac{w^2}{4}\right) \nonumber\\
&\quad\quad- \Gamma\left(-\frac{z}{2}\right) \zeta(-z) \pi^{\frac{z}{2}}  \a^{-\frac{z}{2}-1} \,_1F_1\left(\frac{1+z}{2};\frac{1}{2};\frac{w^2}{4} \right)   
\Bigg)\nonumber\\
&= \sqrt{\b}  \Bigg( 
4 \sum_{n=1}^{\infty} \sigma_{-z}(n) n^{\frac{z}{2}}  e^{\frac{w^2}{4}}  K_{\frac{z}{2},w}(2 n \pi \b)
- \Gamma\left(\frac{z}{2}\right) \zeta(z) \pi^{-\frac{z}{2}} \b^{\frac{z}{2}-1} \,_1F_1\left(\frac{1-z}{2};\frac{1}{2};-\frac{w^2}{4}\right) \nonumber\\
&\quad\quad\quad- \Gamma\left(-\frac{z}{2}\right) \zeta(-z)  \pi^{\frac{z}{2}} \b^{-\frac{z}{2}-1} \,_1F_1\left(\frac{1+z}{2};\frac{1}{2};-\frac{w^2}{4} \right) 
\Bigg).  
\end{align}
\end{theorem}
It is easy to obtain \eqref{genrgeqn} from \eqref{genrgeqeqn} by writing $\a=1/\b$ and $\b=1/\a$ on the right side of \eqref{genrgeqeqn} and then substituting $a=\pi\a$ and $b=\pi\b$. The advantage of \eqref{genrgeqeqn} over \eqref{genrgeqn} is, not only do we see that both sides of \eqref{genrgeqeqn} are even functions of $z$ (which is also the case in \eqref{genrgeqn}) but from \eqref{genrgeqeqn} we infer also that its left side is invariant under the simultaneous replacement of $\a$ by $\b$ and $w$ by $iw$. Thus \eqref{genrgeqeqn} is a \textit{modular-type transformation} of the form $F(z, w, \a)=F(z, iw, \b)$ for $\a\b=1$.\\

\textbf{Remark 1.} In \cite[p.~60]{cohen}, it is shown that the Ramanujan-Guinand formula in \eqref{mainagain} is equivalent to the functional equation of the non-holomorphic Eisenstein series on $SL_{2}(\mathbb{Z})$. (See also \cite[p.~23]{bls} for discussion on this topic.) The generalization of the Ramanujan-Guinand formula that we have obtained in Theorem \ref{genrg} (or, equivalently, in Theorem \ref{genrgeq}) now poses a very interesting question - is this generalization equivalent to the functional equation of some generalization of the non-holomorphic Eisenstein series on $SL_{2}(\mathbb{Z})$?

\textbf{Remark 2.} Even though we have assumed $\a$ and $\b$ to be positive in \eqref{genrgeqn} and \eqref{genrgeqeqn}, these formulas can be extended by analytic continuation to $\alpha, \, \beta \in \Omega \subset \mathbb{C}$, with 
$\mathbb{R} \subset \Omega$, for some region $\Omega$ in the complex plane.

As a special case when $z\to 0$ of Theorem \ref{genrgeq}, we obtain a one-variable generalization of a famous formula of Koshliakov \cite{koshliakov} (see also \cite{bls}).
\begin{corollary}\label{genrgeq0}
Let $w\in\mathbb{C}$. For $\a,\b>0$ such that $\a\b = 1$,
\begin{align}\label{genrgeqeqn0}
&\sqrt{\a}  \Bigg\{ 
4 \sum_{n=1}^{\infty} d(n)  e^{-\frac{w^2}{4}} K_{0,iw}(2  n \pi\a)-\frac{1}{\a}\left((\g-\log(4\pi\a))\left(1+\frac{w^2}{4}\right)+\frac{w^2}{2}\right)\Bigg\}\nonumber\\
&=\sqrt{\b}  \Bigg\{ 
4 \sum_{n=1}^{\infty} d(n)  e^{\frac{w^2}{4}} K_{0,w}(2  n \pi\b)-\frac{1}{\b}\left((\g-\log(4\pi\b))\left(1-\frac{w^2}{4}\right)-\frac{w^2}{2}\right)\Bigg\}.
\end{align}
\end{corollary}
Koshliakov's formula is a special case of the above corollary when $w=0$. Note that the above transformation is of the form $F(w,\a)=F(iw,\b)$ for $\a\b=1$. Like the general theta transformation formula \eqref{eqsym0}, the transformations in \eqref{genrgeqeqn} and \eqref{genrgeqeqn0}, being of the forms $F(z, w, \a)=F(z, iw, \b)$ and $F(w,\a)=F(iw,\b)$ respectively for $\a\b=1$, are also analogues of the general theta transformation but they come from a higher level in the sense that while the general theta transformation formula is associated to the inverse Mellin transform containing the first power of the gamma function, these formulas are associated to the inverse Mellin transform containing the second power, or more generally, of $\G\left(\frac{s-z}{2}\right)\G\left(\frac{s+z}{2}\right)$. This justifies the latter half of the title of this paper.\\ 

So far we know only one representation of $K_{z,w}(x)$, namely, its definition in \eqref{kzw} as an inverse Mellin transform. An introduction of any new special function is useful if it is shown to satisfy a rich theory. We now try to convince the reader that this is indeed the case with $K_{z,w}(x)$. We begin with an integral representation for $K_{z,w}(x)$:
\begin{theorem}\label{kzwint1}
 For $z,w\in\mathbb{C}$ and $\left|\arg x\right|<\frac{\pi}{4}$, we have
\begin{equation}\label{kzwint1eqn}
K_{z,w}(2x) = x^{-z} \int_{0}^\infty e^{-t^2-\frac{x^2}{t^2}} \cos(wt) \cos\left(\frac{wx}{t}\right) t^{2z-1} \,dt.
\end{equation}
\end{theorem}
Basset's formula for the modified Bessel function of the second kind \cite[p.~172]{watson-1944a} is given by
\begin{equation}\label{basset0}
K_{z}(xy)=\frac{\G\left(z+\frac{1}{2}\right)(2x)^{z}}{y^{z}\G(\frac{1}{2})}\int_{0}^{\infty}\frac{\cos(yu)\, du}{(x^2+u^2)^{z+\tfrac{1}{2}}},
\end{equation}
valid for Re$(z)>-\tfrac{1}{2}, y>0$, and $|\arg x|<\tfrac{1}{2}\pi$. The function $K_{0,w}(x)$ admits an elegant Basset-type representation given below.
\begin{theorem}\label{bassetgen}
For $|\arg x|<\tfrac{1}{4}\pi$ and $w\in\mathbb{C}$, we have
\begin{align}\label{bassetgeneqn}
K_{0,w}(x)=\int_{0}^{\infty}\textup{exp}\left(-\frac{w^2x^2}{2(x^2+u^2)}\right)\cos\left(\frac{w^2xu}{2(x^2+u^2)}\right)\frac{\cos u\, du}{\sqrt{x^2+u^2}}.
\end{align}
\end{theorem}
It is easy to see that when $w=0$, the above formula reduces to the special case of \eqref{basset0} when $z=0$ and $y=1$.

We also obtain an interesting bilateral series representation for $K_{z,w}(x)$ involving three different Bessel functions:
\begin{theorem}\label{kzwser}
For $-\tfrac{1}{2}<$ \textup{Re}$(z)<\tfrac{1}{2}$, $w\in\mathbb{C}$, and $|\arg x|<\frac{\pi}{4}$,
\begin{align}\label{kzwsereqn}
K_{z,w}(2x)=\frac{1}{2}\sum_{n=-\infty}^{\infty}(-1)^nK_{z+n}(2x)\left(I_{2n}(2w\sqrt{x})+J_{2n}(2w\sqrt{x})\right).
\end{align}
\end{theorem}
When $w=0$, the above result reduces to the trivial relation $K_{z}(2x)=K_{z}(2x)$ as $I_{0}(0)=J_{0}(0)=1$ and $I_{n}(0)=J_{n}(0)=0$ for $n\neq0$.\\

It is well-known \cite[p.~236]{temme}\footnote{There is minor misprint in that $(2x)^{z}$ is typed as $(2/x)^z$.} that $K_{z}(x)$ can be expressed as a Laplace transform of an elementary function, that is, for Re$(z)>-\frac{1}{2}$ and $|\arg x|<\frac{\pi}{2}$,
\begin{equation}\label{kzlapeqn}
K_{z}(x)=\frac{\sqrt{\pi}(2x)^{z}e^{-x}}{\G\left(z+\frac{1}{2}\right)}\int_{0}^{\infty}e^{-2xt}t^{z-\frac{1}{2}}(t+1)^{z-\frac{1}{2}}\, dt.
\end{equation}
For $K_{z,w}(x)$, we obtain a generalization of the above result which represents $K_{z,w}(x)$ as an infinite series of Laplace transform of a special function.
\begin{theorem}\label{kzwlap}
Let $w\in\mathbb{C}$, \textup{Re}$(z)>-\frac{1}{2}$ and $|\arg x|<\frac{\pi}{4}$. Then
\begin{align}\label{kzwlapeqn}
K_{z,w}(x)&=\frac{(2x)^{z+\frac{1}{2}}}{\G\left(z+\frac{1}{2}\right)}\sum_{n=0}^{\infty}\frac{\left(-\frac{w^2x}{2}\right)^{n}}{(2n)!}\int_{0}^{\infty}t^{z-\frac{1}{2}}(t+1)^{z-\frac{1}{2}}(2t+1)^{-n+\frac{1}{2}}K_{n+\frac{1}{2}}(x(2t+1))\nonumber\\
&\qquad\qquad\qquad\qquad\qquad\qquad\times{}_0F_{2}\left(-;\frac{1}{2},\frac{1}{2}+z;-\frac{w^2x^2t(t+1)}{4}\right)\, dt.
\end{align}
\end{theorem}
The integral in the above representation is indeed a Laplace transform as \cite[p.~934, formula \textbf{8.468}]{grn}
\begin{equation}\label{knhalf}
K_{n+\frac{1}{2}}(y)=\sqrt{\frac{\pi}{2y}}e^{-y}\sum_{k=0}^{n}\frac{(n+k)!}{k!(n-k)!(2y)^k}.
\end{equation}
When $w=0$, we recover \eqref{kzlapeqn} since only the $n=0$ term survives.

Next we obtain a double integral representation for $K_{z,w}(x)$.
\begin{theorem}\label{kzwdi}
Let $w\in\mathbb{C}$. For \textup{Re}$(z)>-1$ and $|\arg x|<\frac{\pi}{4}$, 
\begin{align}\label{kzwdieqn}
K_{z,w}(x)&=\frac{1}{2\G(1+z)}\int_{0}^{\infty}\int_{0}^{\infty}\frac{y^zt^{-1/2}}{\sqrt{y+\frac{x}{2}}}\textup{exp}\left(-2\sqrt{\left(t+\frac{x}{2}\right)\left(y+\frac{x}{2}\right)}\right)\nonumber\\
&\quad\quad\quad\quad\quad\quad\quad\times{}_0F_{2}\left(-;\frac{1}{2},1+z;-\frac{w^2xy}{8}\right){}_0F_{2}\left(-;\frac{1}{2},\frac{1}{2};-\frac{w^2xt}{8}\right)\, dt\, dy.
\end{align}
\end{theorem}
When $w=0$, this gives the following double integral representation for $K_z(x)$ which we were unable to find in the literature.
\begin{align*}
K_{z}(x)&=\frac{1}{2\G(1+z)}\int_{0}^{\infty}\int_{0}^{\infty}\frac{y^zt^{-1/2}}{\sqrt{y+\frac{x}{2}}}\textup{exp}\left(-2\sqrt{\left(t+\frac{x}{2}\right)\left(y+\frac{x}{2}\right)}\right)\, dt\, dy.
\end{align*}
The asymptotic expansion of $K_{z,w}(x)$ for large values of $|x|$, obtained by Nico M. Temme, is now given. His proof of this result is given in the Appendix.
\begin{theorem}\label{thm:Kexpansion}
Let the complex variables $w$ and $z$ belong to compact domains, then for large values of $|x|$, $\vert{\rm arg\, }x\vert<\frac14\pi$, we have
the representation
\begin{equation*}
K_{z,w}(2x)=\frac{1}{4}\sqrt{\frac{\pi}{x}}\,e^{-2x}\bigl(\cos(2w\sqrt{x})P-\sin(2w\sqrt{x})Q+e^{-\frac14w^2}R\bigr), 
\end{equation*}
where $P, Q$ and $R$ have the asymptotic expansions
 \begin{equation*}
\begin{array}{ll}
{\displaystyle
P=1+\frac{32z^2-3w^2-8}{128 x}+O\left(x^{-2}\right),}\\ [8pt]
{\displaystyle
Q=\frac{w}{8\sqrt{x}}+O\left(x^{-\frac{3}{2}}\right),}\\[8pt]
{\displaystyle
R=1+\frac{(4z^2-1)(2-w^2)}{32 x}+O\left(x^{-2}\right).}
\end{array}
\end{equation*}
\end{theorem}
For small values of $x$, we obtain the following result of which the first part is proved using  \eqref{kzwint1eqn} and the second using \eqref{kzwsereqn}.
\begin{theorem}\label{kzwsmall}
\textup{(i)} Let $w\in\mathbb{C}$ be fixed. Consider a fixed $z$ such that \textup{Re}$(z)>0$. Let $\mathfrak{D}=\{x\in\mathbb{C}:|\arg x|<\frac{\pi}{4}\}$. Then as $x\to 0$ along any path in $\mathfrak{D}$, we have
\begin{equation}\label{kzwsmalli}
K_{z,w}(x)\sim\frac{1}{2}\G(z)\left(\frac{x}{2}\right)^{-z}{}_1F_{1}\left(z;\frac{1}{2};\frac{-w^2}{4}\right).
\end{equation}
\textup{(ii)} Let $w\in\mathbb{C}$ be fixed. Let $|\arg x|<\frac{\pi}{4}$. As $x\to 0$,
\begin{equation}\label{kzwsmallii}
K_{0,w}(x)\sim -\log x-\frac{w^2}{2}{}_2F_{2}\left(1,1;\frac{3}{2},2;-\frac{w^2}{4}\right).
\end{equation}
\end{theorem}
The result in \eqref{kzwsmallii} shows that, similar to the modified Bessel function $K_{0}(x)$, the function $K_{0,w}(x)$ also has a logarithmic singularity at $x=0$. Note that when $w=0$, the above two results agree with the corresponding ones for $K_z(x)$, see \eqref{kzxasy0} below.

Our last result on $K_{z,w}(x)$ is the following differential-difference equation that it satisfies. 
\begin{theorem}\label{dde}
Let $z, w\in\mathbb{C}$ and $|\arg x|<\frac{\pi}{4}$. Then
\begin{align*} 
\frac{d^4}{dw^4}K_{z,w}(2x) &+ 2x \left( \frac{d^2}{dw^2}K_{z+1,w}(2x) + \frac{d^2}{dw^2}K_{z-1,w}(2x) \right)\nonumber\\
&+x^{2} \left( K_{z+2,w}(2x) - 2 K_{z,w}(2x) + K_{z-2,w}(2x) \right)=0.
\end{align*} 
\end{theorem}
Since $K_{z,w}(x)=K_{-z,w}(x)$, one may want to rephrase the above equation symmetrically in $z$ and $-z$.

This paper is organized as follows. We collect the preliminary results and basic properties of gamma function, Bessel functions and the confluent hypergeometric function in Section \ref{prelim}. These are used frequently with or without mention in the sequel. In Section \ref{prop}, we develop the theory of the generalized modified Bessel function $K_{z,w}(x)$ by finding series and integral representations, asymptotic expansions and a differential-difference equation. Theorems \ref{recpairkosh}, \ref{xiintgenrgthm} and \ref{genrgeq} are proved in Section \ref{rggensec}. Finally we conclude the paper with Section \ref{cr} consisting of some remarks and possible future work. 

\section{preliminaries}\label{prelim}

Kummer's first transformation for the confluent hypergeometric function \cite[p.~191, Equation (4.1.11)]{aar} is given by
\begin{equation}\label{kft}
{}_1F_{1}(a;c;z)=e^z{}_1F_{1}(c-a;c;-z).
\end{equation}

If $\mathfrak{G}$ and $\mathfrak{H}$ respectively denote the Mellin transforms of $g$ and $h$ satisfying appropriate conditions, and if the line Re$(s)=c$ lies in the common strip of analyticity of $\mathfrak{G}(1-s)$ and $\mathfrak{H}(s)$, then Parseval's identity \cite[p.~82, Equation (3.1.11)]{kp} gives
\begin{equation}\label{par}
\int_{0}^{\infty}g(x)h(x)\, dx=\frac{1}{2\pi i}\int_{(c)}\mathfrak{G}(1-s)\mathfrak{H}(s)\, ds.
\end{equation}
A variant of this formula is \cite[p.~83, Equation (3.1.13)]{kp}
\begin{equation}\label{par2}
\int_0^\infty  {g(x)h\left( {\frac{t}{x}} \right)\frac{{dx}}{x}}=\frac{1}{{2\pi i}}\int_{(\sigma )} {\mathfrak{G}(s)\mathfrak{H}(s){t^{ - s}}ds}.
\end{equation}
Stirling's formula for $\Gamma(s)$, $s=\sigma+it$, in a vertical strip $\alpha\leq\sigma\leq\beta$ is given by
\begin{equation}\label{strivert}
|\Gamma(s)|=(2\pi)^{\tf{1}{2}}|t|^{\sigma-\tf{1}{2}}e^{-\tf{1}{2}\pi |t|}\left(1+O\left(\frac{1}{|t|}\right)\right),
\end{equation}
as $|t|\to\infty$. The reflection formula (along with a
variant), and Legendre's duplication formula for the Gamma function
$\G(s)$ are respectively given by
\begin{align}
 \G(s)\G(1-s)&=\frac{\pi}{\sin(\pi s)},\nonumber\\
 \G\left(\frac{1}{2}+s\right)\G\left(\frac{1}{2}-s\right)&=\frac{\pi}{\cos(\pi s)},\label{ref2}\\
 \G(s)\G\left(s+\frac{1}{2}\right)&=\frac{\sqrt{\pi}}{2^{2s-1}}\G(2s).\label{dup}
\end{align}
For $c=$ Re $s>0$ and Re$(a)>0$, we have \cite[p.~47, Equation 5.30]{ober}
\begin{equation}\label{invmel}
e^{-at^2}\cos bt=\frac{1}{2\pi i}\int_{(c)}\frac{1}{2}a^{-\tfrac{s}{2}}\Gamma\left(\frac{s}{2}\right)e^{-\frac{b^2}{4a}}{}_1F_{1}\left(\frac{1-s}{2};\frac{1}{2};\frac{b^2}{4a}\right)t^{-s}\, ds,
\end{equation}
The Riemann zeta function satisfies the functional equation \cite[p.~59]{dav}
\begin{equation}\label{fe}
\pi^{-\frac{1}{2}s}\G\left(\frac{s}{2}\right)\zeta(s)=\pi^{-\frac{1}{2}(1-s)}\G\left(\frac{1-s}{2}\right)\zeta(1-s).
\end{equation}
For large values of $|x|$ with $|\arg x|<\frac{3\pi}{2}$, we have \cite[p.~238, Equation (9.49)]{temme}
\begin{align*}
K_{z}(x)=\sqrt{\frac{\pi}{2x}}e^{-x}\sum_{n=0}^{\infty}\frac{(z,n)}{(2x)^n},
\end{align*}
where 
\begin{equation*}
(z,n)=\frac{\Gamma(z+n+1/2)}{\Gamma(n+1)\Gamma(z-n+1/2)},
\end{equation*}
implying, in particular, that 
\begin{equation*}
K_{z}(x)\sim\sqrt{\frac{\pi}{2x}}e^{-x}.
\end{equation*}

\section{Properties of the generalized modified Bessel function $K_{z,w}(x)$}\label{prop}
\begin{proof}[Theorem \textup{\ref{kzwint1}}][]
Use \eqref{kft} in \eqref{invmel} and then use the resultant with $s$ replaced by $s-z$, $a=1$ and $b=w$ to find that for $c=$ Re$(s)>$ Re$(z)$,
\begin{align*}
\frac{1}{2\pi i}\int_{(c)}\frac{1}{2}\G\left(\frac{s-z}{2}\right){}_1F_{1}\left(\frac{s-z}{2};\frac{1}{2};-\frac{w^2}{4}\right)t^{-s}\, ds=t^{-z}e^{-t^2}\cos(wt).
\end{align*}
Similarly for $c=$ Re$(s)>-$Re$(z)$,
\begin{align*}
\frac{1}{2\pi i}\int_{(c)}\frac{1}{2}\G\left(\frac{s+z}{2}\right){}_1F_{1}\left(\frac{s+z}{2};\frac{1}{2};-\frac{w^2}{4}\right)t^{-s}\, ds=t^{z}e^{-t^2}\cos(wt).
\end{align*}
So for $c=$ Re$(s)>\pm$Re$(z)$ and Re$(x^2)>0$, that is, for $\left|\arg x\right|<\frac{\pi}{4}$, the above two equations along with \eqref{par2} and \eqref{kzw} imply
\begin{align*}
K_{z,w}(2x)&=\int_{0}^{\infty}t^{z}e^{-t^2}\cos(wt)\left(\frac{x}{t}\right)^{-z}e^{-\frac{x^2}{t^2}}\cos\left(\frac{wx}{t}\right)\frac{dt}{t}\nonumber\\
&=x^{-z}\int_{0}^{\infty}t^{2z-1}e^{-t^2-\frac{x^2}{t^2}}\cos(wt)\cos\left(\frac{wx}{t}\right)\, dt.
\end{align*}
\end{proof}
\begin{lemma}\label{inteqlem}
For $\left|\arg x\right|<\frac{\pi}{4}$ and $w\in\mathbb{C}$,
\begin{equation}\label{inteq}
\int_{0}^{\infty}e^{-t^2-x^2/t^2}\cos(wt)\frac{dt}{t}=\int_{0}^{\infty}\textup{exp}\left(-\frac{w^2x^2}{4(x^2+t^2)}\right)\frac{\cos(2t)}{\sqrt{x^2+t^2}}\, dt.
\end{equation}
\end{lemma}
\begin{proof}
Consider the left-hand side. Expanding $\cos(wt)$ into its Taylor series and then interchanging the order of summation and integration because of absolute convergence, we see that
\begin{align*}
\int_{0}^{\infty}e^{-t^2-x^2/t^2}\cos(wt)\frac{dt}{t}&=\sum_{n=0}^{\infty}\frac{(-w^2)^n}{(2n)!}\int_{0}^{\infty}t^{2n-1}e^{-t^2-x^2/t^2}\, dt\nonumber\\
&=\frac{1}{2}\sum_{n=0}^{\infty}\frac{(-w^2)^n}{(2n)!}\int_{0}^{\infty}u^{n-1}e^{-u-x^2/u}\, du.
\end{align*}
From \cite[p.~344, Formula \textbf{2.3.16.1}]{pbm1}, for Re$(p)>0$, Re$(q)>0$,
\begin{equation}\label{pbm1bes}
\int_{0}^{\infty}y^{s-1}e^{-py-q/y}\, dy=2\left(\frac{q}{p}\right)^{s/2}K_{s}(2\sqrt{pq}).
\end{equation}
This gives for Re$(x^2)>0$, that is, for $\left|\arg x\right|<\frac{\pi}{4}$,
\begin{align}\label{ch1}
\int_{0}^{\infty}e^{-t^2-x^2/t^2}\cos(wt)\frac{dt}{t}=\sum_{n=0}^{\infty}\frac{(-w^2x)^n}{(2n)!}K_{n}(2x).
\end{align}
On the other hand, expanding the exponential function in the integrand on the right side of \eqref{inteq}, separating the $n=0$ term and then interchanging the order of summation and integration because of absolute convergence \cite[p.~30, Thm. 2.1]{temme}, we find that
\begin{align*}
&\int_{0}^{\infty}\textup{exp}\left(-\frac{w^2x^2}{4(x^2+t^2)}\right)\frac{\cos(2t)}{\sqrt{x^2+t^2}}\, dt\nonumber\\
&=\int_{0}^{\infty}\frac{\cos(2t)}{\sqrt{x^2+t^2}}\, dt+\sum_{n=1}^{\infty}\frac{(-w^2x^2)^n}{n!4^n}\int_{0}^{\infty}\frac{\cos(2t)}{(x^2+t^2)^{n+\frac{1}{2}}}\, dt.
\end{align*}
It is to be noted that the first integral on the above right-hand side is not absolutely convergent which is why we need to first separate it before interchanging the order. Employing \eqref{basset0} and making use of the fact that $(\frac{1}{2})_n=(2n)!/(n!4^n)$, we arrive at
\begin{align}\label{ch2}
\int_{0}^{\infty}\textup{exp}\left(-\frac{w^2x^2}{4(x^2+t^2)}\right)\frac{\cos(2t)}{\sqrt{x^2+t^2}}\, dt=K_{0}(2x)+\sum_{n=1}^{\infty}\frac{(-w^2x)^n}{(2n)!}K_{n}(2x).
\end{align}
The identity in the lemma follows immediately from \eqref{ch1} and \eqref{ch2}.
\end{proof}
\begin{proof}[Theorem \textup{\ref{bassetgen}}][]
From \cite[p.~121, Eqn. (43)]{burchnallchaundy}, we have
\begin{equation}\label{acxacy}
{}_1F_{1}(a;c;u){}_1F_{1}(a;c;v)=\sum_{n=0}^{\infty}\frac{(a)_n(c-a)_n}{n!(c)_n(c)_{2n}}(-uv)^n{}_1F_{1}(a+n;c+2n;u+v).
\end{equation}
First assume $x>0$. Let 
\begin{align*}
I(x,w)&:=\frac{1}{2\pi i} \int_{(c)}\Gamma^{2}\bigg(\frac{s}{2}\bigg)\,_1F_1^{2}\bigg(\frac{s}{2};\frac{1}{2};\frac{-w^2}{4}\bigg)  
 x^{-s}ds
\end{align*}
so that from \eqref{kzw},
\begin{equation}\label{kzwi}
K_{0,w}(x)=\frac{1}{4}I\left(\frac{x}{2},w\right).
\end{equation}
Now from \eqref{acxacy}, we have
\begin{align*}
I(x,w)&=\frac{1}{2\pi i} \int_{(c)}\Gamma^{2}\bigg(\frac{s}{2}\bigg)\sum_{n=0}^{\infty}\frac{\left(\tfrac{s}{2}\right)_n\left(\tfrac{1-s}{2}\right)_n}{n!\left(\tfrac{1}{2}\right)_n\left(\tfrac{1}{2}\right)_{2n}}\left(-\frac{w^4}{16}\right)^{n}{}_1F_{1}\left(\frac{s}{2}+n;\frac{1}{2}+2n;-\frac{w^2}{2}\right)x^{-s}\, ds\nonumber\\
&=:I_{1}(x,w)+I_{2}(x,w),
\end{align*}
where
\begin{align}
I_{1}(x,w)&:=\frac{1}{2\pi i} \int_{(c)}\Gamma^{2}\bigg(\frac{s}{2}\bigg){}_1F_{1}\left(\frac{s}{2};\frac{1}{2};-\frac{w^2}{2}\right)x^{-s}\, ds,\nonumber\\
I_{2}(x,w)&:=\frac{1}{2\pi i} \int_{(c)}\Gamma^{2}\bigg(\frac{s}{2}\bigg)\sum_{n=1}^{\infty}\frac{\left(\tfrac{s}{2}\right)_n\left(\tfrac{1-s}{2}\right)_n}{n!\left(\frac{1}{2}\right)_n\left(\frac{1}{2}\right)_{2n}}\left(-\frac{w^4}{16}\right)^{n}{}_1F_{1}\left(\frac{s}{2}+n;\frac{1}{2}+2n;-\frac{w^2}{2}\right)x^{-s}\, ds.\label{i2}
\end{align}
We first evaluate $I_1(x,w)$. First employing \eqref{kft} in \eqref{invmel}, and then using the resultant with $a=2$ and $b=2w$ and $t$ replaced by $t/\sqrt{2}$, we find that for $c=$ Re$(s)>0$,
\begin{equation}\label{i11}
\frac{1}{2\pi i} \int_{(c)}\Gamma\bigg(\frac{s}{2}\bigg){}_1F_{1}\left(\frac{s}{2};\frac{1}{2};-\frac{w^2}{2}\right)t^{-s}\, ds=2e^{-t^2}\cos(\sqrt{2}wt).
\end{equation}
Also for $c=$ Re$(s)>0$,
\begin{equation}\label{i12}
\frac{1}{2\pi i}\int_{(c)}\Gamma\left(\frac{s}{2}\right)t^{-s}\, ds=2e^{-t^2}.
\end{equation}
Hence from \eqref{i11}, \eqref{i12} and Parseval's identity \eqref{par2}, we see that for $0<c=$ Re$(s)<1$,
\begin{equation}\label{i1e}
I_{1}(x,w)=4\int_{0}^{\infty}e^{-t^2-\frac{x^2}{t^2}}\cos(\sqrt{2}wt)\frac{dt}{t}.
\end{equation}
We now evaluate $I_{2}(x,w)$. Let $0<c=$ Re$(s)<1$. The exponential decay of the gamma function, seen from \eqref{strivert}, allows us to interchange the order of summation and integration on the right side of \eqref{i2}. Hence
\begin{align*}
I_{2}(x,w)=\sum_{n=1}^{\infty}\frac{(-w^4/16)^n}{n!\left(\frac{1}{2}\right)_{n}\left(\frac{1}{2}\right)_{2n}}A_n(x,w),
\end{align*}
where
\begin{align*}
A_n(x,w):=\frac{1}{2\pi i}\int_{(c)}\frac{\G\left(\frac{s}{2}\right)\G\left(\frac{s}{2}+n\right)\G\left(\frac{1-s}{2}+n\right)}{\G\left(\frac{1-s}{2}\right)}{}_1F_{1}\left(\frac{s}{2}+n;\frac{1}{2}+2n;-\frac{w^2}{2}\right)x^{-s}\, ds.
\end{align*}
Now write the ${}_1F_{1}$ in the above equation in the form of series and again interchange the order of summation and integration using \eqref{strivert} to arrive at
\begin{align}\label{anxw}
A_n(x,w)=\G\left(\frac{1}{2}+2n\right)\sum_{m=0}^{\infty}\frac{(-w^2/2)^m}{m!}B_{n,m}(x),
\end{align}
where
\begin{align}\label{bnmx}
B_{n,m}(x):=\frac{1}{2\pi i}\int_{(c)}\frac{\G\left(\frac{s}{2}\right)}{\G\left(\frac{1-s}{2}\right)}\frac{\G\left(\frac{1-s}{2}+n\right)\G\left(\frac{s}{2}+n+m\right)}{\G\left(\frac{1}{2}+2n+m\right)}x^{-s}\, ds.
\end{align}
We now evaluate $B_{n,m}(x)$. Using elementary properties of the gamma function, one can show \cite[p.~73]{dav} that
\begin{equation*}
\frac{\G\left(\frac{s}{2}\right)}{\G\left(\frac{1-s}{2}\right)}=\pi^{-\frac{1}{2}}2^{1-s}\G(s)\cos\left(\frac{\pi s}{2}\right).
\end{equation*}
Hence for $0<c=$ Re$(s)<1$,
\begin{align}\label{cosmel}
\frac{1}{2\pi i}\int_{(c)}\frac{\G\left(\frac{s}{2}\right)}{\G\left(\frac{1-s}{2}\right)}x^{-s}\,ds&=\frac{2}{\sqrt{\pi}}\frac{1}{2\pi i}\int_{(c)}\G(s)\cos\left(\frac{\pi s}{2}\right)(2x)^{-s}\, ds\nonumber\\
&=\frac{2}{\sqrt{\pi}}\cos(2x),
\end{align}
as can be seen from \cite[p.~42, Eqn. 5.2]{ober}. Next, Euler's beta integral gives for $0<d=$ Re$(s)<$ Re$(z)$,
\begin{equation*}
\frac{1}{2\pi i}\int_{(d)}\frac{\G(s)\G(z-s)}{\G(z)}x^{-s}\, ds=\frac{1}{(1+x)^z},
\end{equation*}
so that for $-2n-2m<c=$ Re$(s)<1+2n$, $n,\in\mathbb{N}, m\in\mathbb{N}\cup\{0\}$,
\begin{equation}\label{betmel}
\frac{1}{2\pi i}\int_{(c)}\frac{\G\left(\frac{1-s}{2}+n\right)\G\left(\frac{s}{2}+n+m\right)}{\G\left(\frac{1}{2}+2n+m\right)}x^{-s}\, ds=\frac{2x^{2n+2m}}{(1+x^2)^{\frac{1}{2}+2n+m}}.
\end{equation}
From \eqref{cosmel}, \eqref{betmel}, \eqref{bnmx} and \eqref{par2}, we deduce that for $0<$ Re$(s)<1$,
\begin{align*}
B_{n,m}(x)=\frac{4x^{2n+2m}}{\sqrt{\pi}}\int_{0}^{\infty}\frac{t^{2n}\cos(2t)\, dt}{(x^2+t^2)^{\frac{1}{2}+2n+m}},
\end{align*}
which implies through \eqref{anxw},
\begin{align}\label{anxwe}
A_{n}(x,w)=\frac{4x^{2n}}{\sqrt{\pi}}\G\left(\frac{1}{2}+2n\right)\sum_{m=0}^{\infty}\frac{(-w^2x^2/2)^{m}}{m!}\int_{0}^{\infty}\frac{t^{2n}\cos(2t)\, dt}{(x^2+t^2)^{\frac{1}{2}+2n+m}}.
\end{align}
Note that $\sum_{m=0}^{\infty}\frac{(-w^2x^2/2)^{m}}{m!(x^2+t^2)^{\frac{1}{2}+m}}$ converges uniformly on any compact interval of $(0,\infty)$ to $\textup{exp}\left(-\frac{w^2x^2}{2(x^2+t^2)}\right)$. Moreover, it is easy to see that 
\begin{equation*}
\sum_{m=0}^{\infty}\int_{0}^{\infty}\left|\frac{t^{2n}\cos(2t)(-w^2x^2/2)^{m}}{m!(x^2+t^2)^{\frac{1}{2}+2n+m}}\, dt\right|
\end{equation*}
is finite. Then, \cite[p.~30, Thm. 2.1]{temme} permits us to interchange the order of summation and integration in \eqref{anxwe} so that
\begin{align}\label{anxwe1}
A_{n}(x,w)&=\frac{4x^{2n}}{\sqrt{\pi}}\G\left(\frac{1}{2}+2n\right)\int_{0}^{\infty}\frac{t^{2n}\cos(2t)}{(x^2+t^2)^{\frac{1}{2}+2n}}\sum_{m=0}^{\infty}\frac{(-w^2x^2/2)^{m}}{m!(x^2+t^2)^{m}}\, dt\nonumber\\
&=\frac{4x^{2n}}{\sqrt{\pi}}\G\left(\frac{1}{2}+2n\right)\int_{0}^{\infty}\frac{t^{2n}\cos(2t)}{(x^2+t^2)^{\frac{1}{2}+2n}}\textup{exp}\left(-\frac{w^2x^2}{2(x^2+t^2)}\right)\, dt.
\end{align}
Noe \eqref{i2} and \eqref{anxwe1} imply
\begin{align*}
I_{2}(x,w)=4\sum_{n=1}^{\infty}\frac{\left(\frac{-w^4x^2}{16}\right)^n}{n!\left(\frac{1}{2}\right)_n}\int_{0}^{\infty}\frac{t^{2n}\cos(2t)}{(x^2+t^2)^{\frac{1}{2}+2n}}\textup{exp}\left(-\frac{w^2x^2}{2(x^2+t^2)}\right)\, dt.
\end{align*}
Note that $\sum_{n=1}^{\infty}\frac{\left(-w^4x^2/16\right)^n}{n!\left(\frac{1}{2}\right)_n(x^2+t^2)^{2n}}$ converges uniformly to $\cos\left(\frac{w^2xt}{2(x^2+t^2)}\right)-1$ on compact intervals of $(0,\infty)$ and 
\begin{equation*}
\sum_{n=1}^{\infty}\int_{0}^{\infty}\left|\frac{\left(-w^4x^2/16\right)^n}{n!\left(\frac{1}{2}\right)_n}\frac{t^{2n}\cos(2t)}{(x^2+t^2)^{\frac{1}{2}+2n}}\textup{exp}\left(-\frac{w^2x^2}{2(x^2+t^2)}\right)\, dt\right|
\end{equation*}
is finite. Hence another appeal to \cite[p.~30, Thm. 2.1]{temme} allows us to interchange the order of summation and integration so that
\begin{align}\label{i2e}
I_{2}(x,w)=4\int_{0}^{\infty}\textup{exp}\left(-\frac{w^2x^2}{2(x^2+t^2)}\right)\frac{\cos(2t)}{\sqrt{x^2+t^2}}\left(\cos\left(\frac{w^2xt}{2(x^2+t^2)}-1\right)\right)\, dt,
\end{align}
so that from \eqref{i1e} and \eqref{i2e}, we finally arrive at
\begin{align*}
I(x,w)&=4\int_{0}^{\infty}e^{-t^2-\frac{x^2}{t^2}}\cos(\sqrt{2}wt)\frac{dt}{t}\nonumber\\
&\quad+4\int_{0}^{\infty}\textup{exp}\left(-\frac{w^2x^2}{2(x^2+t^2)}\right)\frac{\cos(2t)}{\sqrt{x^2+t^2}}\left(\cos\left(\frac{w^2xt}{2(x^2+t^2)}-1\right)\right)\, dt\nonumber\\
&=4\int_{0}^{\infty}\textup{exp}\left(-\frac{w^2x^2}{2(x^2+t^2)}\right)\frac{\cos(2t)}{\sqrt{x^2+t^2}}\cos\left(\frac{w^2xt}{2(x^2+t^2)}\right)\, dt,
\end{align*}
as can be seen from Lemma \ref{inteqlem}. From \eqref{kzwi}, we now obtain \eqref{bassetgeneqn} upon change of variable. This completes the proof of Theorem \ref{bassetgen} for $x>0$. Since both sides of \eqref{bassetgeneqn} are analytic when $|\arg x|<\pi/4$, the identity holds for $|\arg x|<\pi/4$ by analytic continuation.
\end{proof}
To prove Theorem \ref{kzwser}, we begin with a lemma.
\begin{lemma}\label{kzwdouser}
For $z, w\in\mathbb{C}$ and $|\arg x|<\frac{\pi}{4}$, we have
\begin{align*}
K_{z,w}(2x) =  \sum_{n=0}^{\infty} \sum_{m=0}^{\infty} \frac{(-w^2x)^{n+m}}{(2n)! (2m)!}.
K_{n-m+z}(2x)
\end{align*}
\end{lemma}
\begin{proof}
Using Theorem \ref{kzwint1}, expanding each of the cosines into its Taylor series and then interchanging the order of summation and integration each time, we arrive at
\begin{align*}
K_{z,w}(2x) &= x^{-z} \int_{0}^\infty e^{-t^2-\frac{x^2}{t^2}} \cos(wt) \cos\left(\frac{wx}{t}\right) t^{2z-1} \,dt \\
&= x^{-z} \int_{0}^\infty \sum_{n=0}^{\infty} \frac{(-w^2)^{n}}{(2n)!}   e^{-t^2-\frac{x^2}{t^2}} \cos\left(\frac{wx}{t}\right) t^{2n+2z-1} \,dt\\
&= x^{-z} \sum_{n=0}^{\infty} \frac{(-w^2)^{n}}{(2n)!}  \int_{0}^\infty e^{-t^2-\frac{x^2}{t^2}} \cos\left(\frac{wx}{t}\right) t^{2n+2z-1} \,dt\\
&= x^{-z} \sum_{n=0}^{\infty} \frac{(-w^2)^{n}}{(2n)!}  \int_{0}^\infty \sum_{m=0}^{\infty} \frac{(-w^2x^2)^{m}}{(2m)!} e^{-t^2-\frac{x^2}{t^2}}  t^{2n-2m+2z-1} \,dt\\
&= x^{-z} \sum_{n=0}^{\infty} \frac{(-w^2)^{n}}{(2n)!} \sum_{m=0}^{\infty} \frac{(-w^2x^2)^{m}}{(2m)!}  \int_{0}^\infty e^{-t^2-\frac{x^2}{t^2}}  t^{2(n-m+z)-1} \,dt\\
&= \sum_{n=0}^{\infty} \sum_{m=0}^{\infty} \frac{(-w^2x)^{n+m}}{(2n)! (2m)!} K_{n-m+z}(2x),
\end{align*}
where in the last step we used \eqref{pbm1bes}.
\end{proof}
We are now ready to prove Theorem \ref{kzwser}.
\begin{proof}[Theorem \textup{\ref{kzwser}}][]
Let $x\in\mathbb{C}$ such that $|\arg x|<\pi/4$ and $w\in\mathbb{C}$. By Basset's formula \eqref{basset0} and the fact that $K_{\nu}(x)$ is an even function of its order, we have
\begin{equation}\label{basset2}
K_{n-m+z}(2x)=\begin{cases}
\displaystyle\frac{\G(n-m+z+\frac{1}{2})(2x)^{n-m+z}}{2^{n-m+z}\G(\frac{1}{2})}\int_{0}^{\infty}\frac{\cos(2u)\, du}{(u^2+x^2)^{n-m+z+\frac{1}{2}}},\hspace{1mm}\text{if }\hspace{1mm}\text{Re}(n-m+z)\geq-\frac{1}{2},\\
\displaystyle\frac{\G(m-n-z+\frac{1}{2})(2x)^{m-n-z}}{2^{m-n-z}\G(\frac{1}{2})}\int_{0}^{\infty}\frac{\cos(2u)\, du}{(u^2+x^2)^{m-n-z+\frac{1}{2}}},\hspace{1mm}\text{if }\hspace{1mm}\text{Re}(m-n-z)\geq-\frac{1}{2}.
\end{cases}
\end{equation}
By an application of Lemma \ref{kzwdouser},
\begin{align*}
K_{z,w}(2x)=\sum_{n=0}^{\infty}\sum_{m=0}^{n}\frac{(-w^2x)^{n+m}}{(2n)! (2m)!}K_{n-m+z}(2x) +\sum_{n=0}^{\infty}\sum_{m=n+1}^{\infty}\frac{(-w^2x)^{n+m}}{(2n)! (2m)!}
K_{m-n-z}(2x).
\end{align*}
By the hypothesis, $-\frac{1}{2}<$ Re$(z)<\frac{1}{2}$. Now Re$(z)>-1/2$ implies Re$(z)+n+\frac{1}{2}>n$. So if $m\leq n$, then $m<$ Re$(z)+n+\frac{1}{2}$, that, is, Re$(n-m+z)>-\frac{1}{2}$. Also, Re$(z)<\frac{1}{2}$ implies Re$(z)+n-\frac{1}{2}<n$. Hence if $m\geq n+1$, then Re$(z)+n-\frac{1}{2}<m$, that is, Re$(m-n-z)>-\frac{1}{2}$. Hence along with \eqref{basset2}, we find that
\begin{equation}\label{ks1s2}
K_{z,w}(2x)=S_{1}(z,w,x)+S_{2}(z,w,x),
\end{equation}
where
\begin{align}
S_1(z,w,x)&=\sum_{n=0}^{\infty}\sum_{m=0}^{n}\frac{(-w^2x)^{n+m}}{(2n)! (2m)!}\frac{\G(n-m+z+\frac{1}{2})x^{n-m+z}}{\G(\frac{1}{2})}\int_{0}^{\infty}\frac{\cos(2u)\, du}{(u^2+x^2)^{n-m+z+\frac{1}{2}}}\\
S_{2}(z,w,x)&=\sum_{n=0}^{\infty}\sum_{m=n+1}^{\infty}\frac{(-w^2x)^{n+m}}{(2n)! (2m)!}\frac{\G(m-n-z+\frac{1}{2})x^{m-n-z}}{\G(\frac{1}{2})}\int_{0}^{\infty}\frac{\cos(2u)\, du}{(u^2+x^2)^{m-n-z+\frac{1}{2}}}\label{s2sum}.
\end{align}
We first simplify $S_1(z,w,x)$. Writing $S_1(z,w,x)$ as a doubly infinite series, we see that
\begin{align}\label{s1zwx}
S_1(z,w,x)&=\sum_{m=0}^{\infty}\sum_{d=0}^{\infty}\frac{(-w^2)^{2m+d}x^{2m+2d+z}}{(2m)!(2m+2d)!}\frac{\G(d+z+\frac{1}{2})}{\G(\frac{1}{2})}\int_{0}^{\infty}\frac{\cos(2u)\, du}{(u^2+x^2)^{d+z+\frac{1}{2}}}\nonumber\\
&=:T_1(z,w,x)+T_2(z,w,x),
\end{align}
where
\begin{align*}
T_1(z,w,x)&=\frac{\G(z+\frac{1}{2})}{\G(\frac{1}{2})}\sum_{m=0}^{\infty}\frac{w^{4m}x^{2m+z}}{((2m)!)^2}\int_{0}^{\infty}\frac{\cos(2u)\, du}{(u^2+x^2)^{z+\frac{1}{2}}}\nonumber\\
T_2(z,w,x)&=\sum_{m=0}^{\infty}\frac{w^{4m}x^{2m+z}}{(2m)!}\int_{0}^{\infty}\frac{\cos(2u)\, du}{(u^2+x^2)^{z+\frac{1}{2}}}\sum_{d=1}^{\infty}\frac{\left(-\frac{w^2x^2}{x^2+u^2}\right)^{d}}{(2m+2d)!}\frac{\G(d+z+\frac{1}{2})}{\G(\frac{1}{2})}.
\end{align*}
Employing \eqref{basset0}, we have
\begin{align}\label{t1e}
T_1(z,w,x)&=\frac{\G(z+\frac{1}{2})}{\G(\frac{1}{2})}\sum_{m=0}^{\infty}\frac{w^{4m}x^{2m+z}}{((2m)!)^2}\frac{\sqrt{\pi}x^{-z}K_{z}(2x)}{\G(z+\frac{1}{2})}\nonumber\\
&=K_{z}(2x)\sum_{m=0}^{\infty}\frac{w^{4m}x^{2m}}{((2m)!)^2}\nonumber\\
&=\frac{1}{2}K_{z}(2x)\left(I_{0}(2w\sqrt{x})+J_{0}(2w\sqrt{x})\right),
\end{align}
where the last step follows from the definitions \eqref{sumbesselj} and \eqref{besseli} of the two Bessel functions.
Now it is easy to see that
\begin{equation*}
\sum_{d=1}^{\infty}\frac{\left(-\frac{w^2x^2}{x^2+u^2}\right)^{d}}{(2m+2d)!}\frac{\G(d+z+\frac{1}{2})}{\G(\frac{1}{2})}=\frac{-w^2x^2}{(u^2+x^2)}\frac{\G(z+\frac{3}{2})}{\G(\frac{1}{2})\G(2m+3)}{}_2F_{2}\left(1,z+\frac{3}{2};m+\frac{3}{2},m+2;\frac{-w^2x^2}{4(u^2+x^2)}\right).
\end{equation*}
To see this, we use the series representation of the right side and apply twice the duplication formula \eqref{dup} for the gamma function to arrive at the left side. Thus,
\begin{align*}
T_2(z,w,x)=-&w^2x^{z+2}\frac{\G(z+\frac{3}{2})}{\G(\frac{1}{2})}\sum_{m=0}^{\infty}\frac{w^{4m}x^{2m}}{(2m)!(2m+2)!}\nonumber\\
&\times\int_{0}^{\infty}\frac{\cos(2u)}{(u^2+x^2)^{z+\frac{3}{2}}}{}_2F_{2}\left(1,z+\frac{3}{2};m+\frac{3}{2},m+2;\frac{-w^2x^2}{4(u^2+x^2)}\right)\, du.
\end{align*}
The case $d=0$ in \eqref{s1zwx} is singled out to guarantee absolute convergence. This then allow us to interchange the order of summation and integration as well as the order of two summations. (Note that we could not have done the interchange had we kept the $d=0$ term in the infinite sum over $d$ in \eqref{s1zwx}. Thus writing the ${}_2F_{2}$ in the form of a series and making the interchanges, we arrive at
\begin{align*}
T_2(z,w,x)&=-w^2x^{z+2}\frac{\G\left(z+\frac{3}{2}\right)}{\G(\frac{1}{2})}\sum_{k=0}^{\infty}\left(z+\frac{3}{2}\right)_k\left(\frac{-w^2x^2}{4}\right)^k\nonumber\\
&\quad\quad\times\int_{0}^{\infty}\frac{\cos(2u)\, du}{(u^2+x^2)^{z+k+\frac{3}{2}}}\sum_{m=0}^{\infty}\frac{w^{4m}x^{2m}}{(2m)!(2m+2)!(m+\frac{3}{2})_k(m+2)_k}
\end{align*}  
After representing the rising factorials in the inner series over $m$ in terms of gamma functions and applying the duplication formula \eqref{dup} for the gamma function, we are led upon simplification to
\begin{align*}
\sum_{m=0}^{\infty}\frac{w^{4m}x^{2m}}{(2m)!(2m+2)!(m+\frac{3}{2})_k(m+2)_k}&=2^{2k}\sum_{m=0}^{\infty}\frac{w^{4m}x^{2m}\G(2m+3)}{(2m)!(2m+2)!\G(2m+2k+3)}\nonumber\\
&=2^{2k}\sum_{m=0}^{\infty}\frac{w^{4m}x^{2m}}{(2m)!\G(2m+2k+3)}\nonumber\\
&=\frac{2^{2k-1}}{w^{2k+2}x^{k+1}}\left(I_{2k+2}(2w\sqrt{x})+J_{2k+2}(2w\sqrt{x})\right),
\end{align*}
where the last step follows, as in \eqref{t1e}, by employing the series definitions of the two Bessel functions and simplifying. Thus,
\begin{align}\label{t2e}
T_2(z,w,x)&=-\frac{x^{z+1}}{2\G(\frac{1}{2})}\sum_{k=0}^{\infty}\G\left(z+k+\frac{3}{2}\right)(-x)^k\left(I_{2k+2}(2w\sqrt{x})+J_{2k+2}(2w\sqrt{x})\right)\nonumber\\
&\quad\quad\times\int_{0}^{\infty}\frac{\cos(2u)\, du}{(u^2+x^2)^{z+k+\frac{3}{2}}}\nonumber\\
&=\frac{1}{2}\sum_{k=0}^{\infty}(-1)^{k+1}K_{k+1+z}(2x)\left(I_{2k+2}(2w\sqrt{x})+J_{2k+2}(2w\sqrt{x})\right),
\end{align}
where in the last step, we applied \eqref{basset0} again. Therefore from \eqref{s1zwx}, \eqref{t1e} and \eqref{t2e}, we see that
\begin{align}\label{s1fe}
S_{1}(z,w,x)=\frac{1}{2}\sum_{k=0}^{\infty}(-1)^{k}K_{k+z}(2x)\left(I_{2k}(2w\sqrt{x})+J_{2k}(2w\sqrt{x})\right).
\end{align}
We still need to evaluate $S_{2}(z,w,x)$. To that end, let $\ell=m-n$ in \eqref{s2sum} so that
\begin{align*}
S_{2}(z,w,x)&=\sum_{n=0}^{\infty}\sum_{\ell=1}^{\infty}\frac{(-w^2)^{2n+\ell}x^{2n+2\ell-z}\G(\ell+\frac{1}{2}-z)}{(2n)!(2n+2\ell)!\G(\frac{1}{2})}\int_{0}^{\infty}\frac{\cos(2u)\, du}{(u^2+x^2)^{\ell-z+\frac{1}{2}}}\nonumber\\
&=S_{1}(-z,w,x)-\sum_{n=0}^{\infty}\frac{(-w^2)^{2n}x^{2n-z}\G(\frac{1}{2}-z)}{((2n)!)^2\G(\frac{1}{2})}\int_{0}^{\infty}\frac{\cos(2u)\, du}{(u^2+x^2)^{-z+\frac{1}{2}}}\nonumber\\
\end{align*}
as can be seen from \eqref{s1zwx}. Since Re$(z)<1/2$, we can employ \eqref{basset0} in the single series over $n$ in the above equation. Then simplifying as in \eqref{t1e} and making use of the fact that $K_{\nu}(\lambda)$ is an even function of $\nu$, we see from the above equation that
\begin{align}\label{s2feb}
S_{2}(z,w,x)&=S_{1}(-z,w,x)-\frac{1}{2}K_{z}(2x)\left(I_{0}(2w\sqrt{x})+J_{0}(2w\sqrt{x})\right)\nonumber\\
&=\frac{1}{2}\sum_{k=1}^{\infty}(-1)^{k}K_{k-z}(2x)\left(I_{2k}(2w\sqrt{x})+J_{2k}(2w\sqrt{x})\right),
\end{align}
where the last step follows from \eqref{s1fe}. Now replace $k$ by $-k$ in \eqref{s2feb}, again make use of the fact that $K_{\nu}(\lambda)$ is an even function of $\nu$ along with the identities $J_{-n}(\lambda)=(-1)^nJ_{n}(\lambda)$ and $I_{-n}(\lambda)=I_{n}(\lambda)$ so as to obtain
\begin{align}\label{s2fe}
S_{2}(z,w,x)=\frac{1}{2}\sum_{k=-\infty}^{-1}(-1)^{k}K_{k+z}(2x)\left(I_{2k}(2w\sqrt{x})+J_{2k}(2w\sqrt{x})\right).
\end{align}
Finally, \eqref{ks1s2}, \eqref{s1fe} and \eqref{s2fe} imply \eqref{kzwsereqn}.
\end{proof}
\begin{proof}[Theorem \textup{\ref{kzwlap}}][]
Replace $x$ by $x/2$ in \eqref{kzwint1eqn} and then let $t=\sqrt{\frac{xu}{2}}$ in the resulting equation to arrive at
\begin{align}\label{kzwalt}
K_{z,w}(x)=\frac{1}{2}\int_{0}^{\infty}\textup{exp}\left(-\frac{x}{2}\left(u+\frac{1}{u}\right)\right)\cos\left(\frac{w\sqrt{xu}}{\sqrt{2}}\right)\cos\left(\frac{w\sqrt{x}}{\sqrt{2u}}\right)u^{-z-1}\, du,
\end{align}
where the last step follows from the fact that $K_{z,w}(x)=K_{-z,w}(x)$. Now using \cite[p.~186, Equation (4.25)]{ober} \footnote{There is a typo in the argument of ${}_0F_{2}$ in the version given there in that the $-\frac{a^2y}{2}$ should be $-\frac{a^2y}{4}$.}, it can be seen that from Re$(z)>-1/2$,
\begin{align}\label{upcos}
u^{-z-\frac{1}{2}}\cos\left(\frac{w\sqrt{x}}{\sqrt{2u}}\right)=\frac{1}{\G\left(z+\frac{1}{2}\right)}\int_{0}^{\infty}e^{-yu}y^{z-\frac{1}{2}}{}_0F_{2}\left(-;\frac{1}{2},\frac{1}{2}+z;-\frac{w^2xy}{8}\right)\, dy.
\end{align}
This result can be easily obtained by writing the ${}_0F_{2}$ as a series and then integrating term by term. Now substitute \eqref{upcos} in \eqref{kzwalt} and interchange the order of integration, which is permissible due to absolute convergence, to arrive at
\begin{align}\label{kzwdouble}
K_{z,w}(x)&=\frac{1}{2\G\left(z+\frac{1}{2}\right)}\int_{0}^{\infty}y^{z-\frac{1}{2}}{}_0F_{2}\left(-;\frac{1}{2},\frac{1}{2}+z;-\frac{w^2xy}{8}\right)\, dy\nonumber\\
&\qquad\qquad\qquad\times\int_{0}^{\infty}\textup{exp}\left(-u\left(y+\frac{x}{2}\right)-\frac{x}{2u}\right)\cos\left(\frac{w\sqrt{xu}}{\sqrt{2}}\right)\frac{du}{\sqrt{u}}.
\end{align}
Arguing as in the first part of Lemma \ref{inteqlem}, we find that
\begin{equation*}
\int_{0}^{\infty}e^{-v^2-\frac{x^2}{v^2}}\cos(wv)v^{2z-1}\, dv=x^z\sum_{n=0}^{\infty}\frac{(-w^2x)^n}{(2n)!}K_{n+z}(2x).
\end{equation*}
In the above equation let $z=1/2$, $v=\sqrt{y+\frac{x}{2}}\sqrt{u}$ and replace $w$ by $\frac{w\sqrt{x/2}}{\sqrt{y+x/2}}$ and $x$ by $\sqrt{\frac{x}{2}\left(y+\frac{x}{2}\right)}$ so that 
\begin{align}\label{inteseries}
&\int_{0}^{\infty}\textup{exp}\left(-u\left(y+\frac{x}{2}\right)-\frac{x}{2u}\right)\cos\left(\frac{w\sqrt{xu}}{\sqrt{2}}\right)\frac{du}{\sqrt{u}}\nonumber\\
&=\frac{2(\frac{x}{2})^{1/4}}{(y+\frac{x}{2})^{1/4}}\sum_{n=0}^{\infty}\frac{1}{(2n)!}\left(\frac{-w^2x^{3/2}}{2^{3/2}\sqrt{y+x/2}}\right)^{n}K_{n+\frac{1}{2}}\left(2\sqrt{\frac{x}{2}\left(y+\frac{x}{2}\right)}\right).
\end{align}
Substituting \eqref{inteseries} in \eqref{kzwdouble} and interchanging the order of summation and integration due to absolute convergence, we see that
\begin{align*}
K_{z,w}(x)&=\frac{(\frac{x}{2})^{1/4}}{\G\left(z+\frac{1}{2}\right)}\sum_{n=0}^{\infty}\frac{\left(\frac{-w^2x^{3/2}}{2^{3/2}}\right)^n}{(2n)!}\nonumber\\
&\times\int_{0}^{\infty}\frac{y^{z-\frac{1}{2}}}{\left(y+\frac{x}{2}\right)^{\frac{n}{2}+\frac{1}{4}}}K_{n+\frac{1}{2}}\left(2\sqrt{\frac{x}{2}\left(y+\frac{x}{2}\right)}\right){}_0F_{2}\left(-;\frac{1}{2},\frac{1}{2}+z;-\frac{w^2xy}{8}\right)\, dy
\end{align*}
Next, make a change of variable $\l=\sqrt{y+\frac{x}{2}}/\sqrt{\frac{x}{2}}$ so as to obtain
\begin{align}\label{kzw1inf}
K_{z,w}(x)&=\frac{x^{z+\frac{1}{2}}2^{-z+\frac{1}{2}}}{\G\left(z+\frac{1}{2}\right)}\sum_{n=0}^{\infty}\frac{\left(\frac{-w^2x}{2}\right)^n}{(2n)!}\nonumber\\
&\times\int_{1}^{\infty}(\l^2-1)^{z-\frac{1}{2}}\l^{-n+\frac{1}{2}}K_{n+\frac{1}{2}}(x\l){}_0F_{2}\left(-;\frac{1}{2},\frac{1}{2}+z;-\frac{w^2x^2(\l^2-1)}{16}\right)\, d\l.
\end{align}
Finally let $\l=2t+1$ to arrive at \eqref{kzwlapeqn}.
\end{proof}
\textbf{Remark 3.} We separately record \eqref{kzw1inf} only because the corresponding special case when $w=0$ is recorded so in many texts as well as papers.
\begin{proof}[Theorem \textup{\ref{kzwdi}}][]
Replace $z$ by $z+\frac{1}{2}$ in \eqref{upcos} so as to have for Re$(z)>-1$,
\begin{align}\label{upcos1}
u^{-z-1}\cos\left(\frac{w\sqrt{x}}{\sqrt{2u}}\right)=\frac{1}{\G\left(z+1\right)}\int_{0}^{\infty}e^{-yu}y^{z}{}_0F_{2}\left(-;\frac{1}{2},1+z;-\frac{w^2xy}{8}\right)\, dy.
\end{align}
Substitute the above equation in \eqref{kzwalt} and interchange the order of integration, which is valid by absolute convergence, so that
\begin{align}\label{kzwdouble1}
K_{z,w}(x)&=\frac{1}{2\G\left(z+1\right)}\int_{0}^{\infty}y^{z}{}_0F_{2}\left(-;\frac{1}{2},1+z;-\frac{w^2xy}{8}\right)\, dy\nonumber\\
&\qquad\qquad\qquad\times\int_{0}^{\infty}\textup{exp}\left(-u\left(y+\frac{x}{2}\right)-\frac{x}{2u}\right)\cos\left(\frac{w\sqrt{xu}}{\sqrt{2}}\right)\, du.
\end{align}
Next, replace $u$ by $1/u$ in \eqref{upcos1} and then let $z=-1/2$ so that 
\begin{align}\label{upcosspl}
\cos\left(\frac{w\sqrt{xu}}{\sqrt{2}}\right)=\frac{1}{\sqrt{\pi u}}\int_{0}^{\infty}e^{-t/u}t^{-1/2}{}_0F_{2}\left(-;\frac{1}{2},\frac{1}{2};-\frac{w^2xt}{8}\right)\, dt.
\end{align}
Now substitute \eqref{upcosspl} in \eqref{kzwdouble1} and again interchange the order of integration. This gives
\begin{align}\label{kzwdouble2}
K_{z,w}(x)&=\frac{1}{2\sqrt{\pi}\G\left(z+1\right)}\int_{0}^{\infty}\int_{0}^{\infty}y^{z}t^{-1/2}{}_0F_{2}\left(-;\frac{1}{2},1+z;-\frac{w^2xy}{8}\right)\nonumber\\
&\quad\times{}_0F_{2}\left(-;\frac{1}{2},\frac{1}{2};-\frac{w^2xt}{8}\right)\int_{0}^{\infty}\textup{exp}\left(-u\left(y+\frac{x}{2}\right)-\left(t+\frac{x}{2}\right)\frac{1}{u}\right)\frac{du}{\sqrt{u}}\, dt\, dy.
\end{align}
Using \eqref{pbm1bes}, the innermost integral is now evaluated to
\begin{align}\label{konehalf}
2\left(\frac{t+\frac{x}{2}}{y+\frac{x}{2}}\right)^{1/4}K_{\frac{1}{2}}\left(2\sqrt{\left(t+\frac{x}{2}\right)\left(y+\frac{x}{2}\right)}\right)&=\frac{\sqrt{\pi}}{\sqrt{y+\frac{x}{2}}}\textup{exp}\left(-2\sqrt{\left(t+\frac{x}{2}\right)\left(y+\frac{x}{2}\right)}\right),
\end{align}
since from \eqref{knhalf}, we have $K_{\frac{1}{2}}(x)=\displaystyle\sqrt{\frac{\pi}{2x}}e^{-x}$. The representation in \eqref{kzwdieqn} now follows from \eqref{kzwdouble2} and \eqref{konehalf}.
\end{proof}
\begin{proof}[Theorem \textup{\ref{kzwsmall}}][]
Replacing $x$ by $x/2$ in \eqref{kzwint1eqn}, we get
\begin{align}\label{kzwint2}
K_{z,w}(x)=\left(\frac{x}{2}\right)^{-z}\int_{0}^{\infty}e^{-t^2-\frac{x^2}{4t^2}} \cos(wt) \cos\left(\frac{wx}{2t}\right) t^{2z-1} \,dt.
\end{align}
Next, for Re$(z)>0$,
\begin{align*}
\lim_{x\to 0}\int_{0}^{\infty}e^{-t^2-\frac{x^2}{4t^2}} \cos(wt) \cos\left(\frac{wx}{2t}\right) t^{2z-1} \,dt&=\int_{0}^{\infty}\lim_{x\to 0}e^{-t^2-\frac{x^2}{4t^2}} \cos(wt) \cos\left(\frac{wx}{2t}\right) t^{2z-1} \,dt\nonumber\\
&=\int_{0}^{\infty}e^{-t^2} \cos(wt) t^{2z-1} \,dt\nonumber\\
&=\frac{1}{2}\G(z){}_1F_{1}\left(z;\frac{1}{2};-\frac{w^2}{4}\right),
\end{align*}
as can be seen from \cite[p.~47, Eqn. 5.30]{ober}. 
The above two equations lead us to \eqref{kzwsmalli} for $x$ lying in the region $\mathfrak{D}$ and tending to $0$. 

To prove $\textup{(ii)}$ of Theorem \eqref{kzwsmall}, we note the following asymptotic formulas for the modified Bessel functions $I_{z}(x)$ and $K_{z}(x)$ as $x\to 0$ \cite[p.~375, equations (9.7.1), (9.7.2)]{stab}:
\begin{equation}\label{izxasy0}
I_{z}(x)\sim\frac{\left(x/2\right)^{z}}{\G(z+1)}, z\neq -1,-2,-3,\cdots.
\end{equation}
and
\begin{equation}\label{kzxasy0}
K_{z}(x)\sim\begin{cases}
 \frac{1}{2}\G(z)\left(\frac{x}{2}\right)^{-z}, \text{if}\hspace{1mm}\textup{Re }z>0,\\
-\log x,\hspace{9mm} \text{if}\hspace{1mm}z=0.
\end{cases}
\end{equation}
From \eqref{kzwsereqn}, for $|\arg x|<\frac{\pi}{4}$,
\begin{align*}
K_{0,w}(x)=\frac{1}{2}K_{0}(x)\left(I_{0}(w\sqrt{2x})+J_{0}(w\sqrt{2x})\right)+\sum_{n=1}^{\infty}(-1)^nK_{n}(x)\left(I_{2n}(w\sqrt{2x})+J_{2n}(w\sqrt{2x})\right).
\end{align*}
Consider the first term on the above right-hand side. Note that as $x\to 0$, $I_{0}(w\sqrt{2x})\to 0$ and $J_{0}(w\sqrt{2x})\to 0$, so along with the second part of \eqref{kzxasy0}, this implies that 
\begin{equation}\label{kzwasy01}
\frac{1}{2}K_{0}(x)\left(I_{0}(w\sqrt{2x})+J_{0}(w\sqrt{2x})\right)\to-\log x.
\end{equation}
Now from \eqref{izxasy0}, as $x\to 0$,
\begin{equation*}
I_{2n}(w\sqrt{2x})\sim\frac{\left(w\sqrt{\frac{x}{2}}\right)^{2n}}{\G(2n+1)}.
\end{equation*}
Also, from \eqref{besseli} and \eqref{izxasy0}, we find that as $x\to 0$,
\begin{equation*}
J_{2n}(w\sqrt{2x})=(-1)^nI_{2n}(-iw\sqrt{2x})\sim(-1)^n\frac{\left(-iw\sqrt{\frac{x}{2}}\right)^{2n}}{\G(2n+1)}.
\end{equation*}
Interchanging the order of limit and summation using \cite[p.~149, Theorem 7.11]{rudin} and combining the above two equations with the first part of \eqref{kzxasy0},
we find that
{\allowdisplaybreaks\begin{align}\label{kzwasy02}
&\lim_{x\to 0}\sum_{n=1}^{\infty}(-1)^nK_{n}(x)\left(I_{2n}(w\sqrt{2x})+J_{2n}(w\sqrt{2x})\right)\nonumber\\
&=\sum_{n=1}^{\infty}\lim_{x\to 0}(-1)^nK_{n}(x)\left(I_{2n}(w\sqrt{2x})+J_{2n}(w\sqrt{2x})\right)\nonumber\\
&=\sum_{n=1}^{\infty}\lim_{x\to 0}\frac{(-1)^n}{2}\G(n)\left(\frac{x}{2}\right)^{-n}\left(\frac{\left(w\sqrt{\frac{x}{2}}\right)^{2n}}{\G(2n+1)}+(-1)^n\frac{\left(-iw\sqrt{\frac{x}{2}}\right)^{2n}}{\G(2n+1)}\right)\nonumber\\
&=\sum_{n=1}^{\infty}\frac{\G(n)}{\G(2n+1)}(-w^2)^n\nonumber\\
&=-\frac{w^2}{2}{}_2F_{2}\left(1,1;\frac{3}{2},2;-\frac{w^2}{4}\right),
\end{align}}
where in the last step we used \eqref{dup}. The required asymptotic formula is obtained from \eqref{kzwasy01} and \eqref{kzwasy02}.
\end{proof}
Before proving Theorem \ref{dde}, we state and prove two simple lemmas.
\begin{lemma}\label{l1}
For $z, w\in\mathbb{C}$ and $|\arg x|<\frac{\pi}{4}$,
\begin{align*}
x^{z}  K_{z,w}(2x) &= \frac{e^{2x}}{2}\int_{0}^\infty e^{- (t+\frac{x}{t})^2} \cos\left(w\left(t+\frac{x}{t}\right)\right) t^{2z-1} \,dt\nonumber\\
 &\quad+ \frac{e^{-2x}}{2}\int_{0}^\infty e^{- (t-\frac{x}{t})^2} \cos\left(w\left(t-\frac{x}{t}\right)\right) t^{2z-1} \,dt.
\end{align*}
\end{lemma}
\begin{proof}
The proof readily follows from \eqref{kzwint1eqn} and the elementary trigonometric identity $2\cos A\cos B=\cos(A+B)+\cos(A-B)$.
\end{proof}
\begin{lemma}\label{l2}
Let $z, w\in\mathbb{C}$ and $|\arg x|<\frac{\pi}{4}$. Let 
\begin{align*}
I(z,w,x) := \frac{e^{2x}}{2} \int_{0}^\infty e^{- (t+\frac{x}{t})^2} \cos\left(w\left(t+\frac{x}{t}\right)\right) t^{2z-1} \,dt.
\end{align*}
Then $$ -\frac{d^2 I(z,w,x)}{dw^2} = I(z+1,w,x) + 2x I(z,w,x) + x^2 I(z-1,w,x).$$
\end{lemma}
\begin{proof}
$$\frac{d}{dw}\left(\cos\left(w\left(t+\frac{x}{t}\right)\right)\right) = -\left(t+\frac{x}{t}\right) \sin\left(w\left(t+\frac{x}{t}\right)\right)$$
$$\frac{d^2}{dw^2}\left(\cos\left(w\left(t+\frac{x}{t}\right)\right)\right) = -\left(t+\frac{x}{t}\right)^2\cos\left(w\left(t+\frac{x}{t}\right)\right)$$
$$-\frac{d^2}{dw^2}\left(\cos\left(w\left(t+\frac{x}{t}\right)\right)\right) = t^2 \cos\left(w\left(t+\frac{x}{t}\right)\right) + 2x \cos\left(w\left(t+\frac{x}{t}\right)\right) + \frac{x^2}{t^2}\cos\left(w\left(t+\frac{x}{t}\right)\right)$$
The identity now follows by differentiating under the integral sign.
\end{proof}
\begin{proof}[Theorem \textup{\ref{dde}}][]
Let $K(z,w,x) = x^{z} K_{z,w}(2x)$. From Lemma \ref{l1},
\begin{equation}\label{eb0}
K(z,w,x) = I(z,w,x) + I(z,w,-x),
\end{equation}
where as Lemma \ref{l2} gives
\begin{align}
\frac{d^2 I(z,w,x)}{dw^2} &= - I(z+1,w,x) - 2x I(z,w,x) - x^2 I(z-1,w,x),\label{eb1}\\
\frac{d^2 I(z,w,-x)}{dw^2} &= - I(z+1,w,-x) + 2x I(z,w,-x) - x^2 I(z-1,w,-x).\label{eb2}
\end{align} 
From \eqref{eb0}, \eqref{eb1} and \eqref{eb2}, we obtain
\begin{equation}\label{eb3}
\frac{d^2 K(z,w,x)}{dw^2} = -K(z+1,w,x) - x^2 K(z-1,w,x) - 2x (I(z,w,x)- I(z,w,-x)) .
\end{equation}
Taking the second derivative with respect to $w$ on both sides of the above equation leads to
\begin{equation}\label{e33}
\frac{d^4 K(z,w,x)}{dw^4} = -\frac{d^2 K(z+1,w,x)}{dw^2} - x^2 \frac{d^2 K(z-1,w,x)}{dw^2} - 2x \left(\frac{d^2 I(z,w,x)}{dw^2}- \frac{d^2 I(z,w,-x)}{dw^2}\right). 
\end{equation}
Using (\ref{eb1}), \eqref{eb2} and (\ref{eb3}) in (\ref{e33}), we arrive at
\begin{multline}\label{e4}
\frac{d^4 K(z,w,x)}{dw^4} = K(z+2,w,x) + 6x^2 K(z,w,x) + x^4 K(z-2,w,x)\\
+ 4x  (I(z+1,w,x)- I(z+1,w,-x)) + 4x^3 (I(z-1,w,x) - I(z-1,w,-x)). 
\end{multline}
Employing (\ref{eb3}) in (\ref{e4}) twice, we get
\begin{align*}
\frac{d^4}{dw^4}K(z,w,x) + 2\frac{d^2}{dw^2}K(z+1,w,x) + 2 x^2\frac{d^2}{dw^2}K(z-1,w,x) \nonumber\\
= -K(z+2,w,x) + 2x^2 K(z,w,x) - x^4 K(z-2,w,x).
\end{align*}
The desired differential-difference equation follows readily by substituting back $K(z,w,x) = x^{z} K_{z,w}(2x)$.
\end{proof}
\section{A pair of functions reciprocal in the Koshliakov kernal and a generalization of the Ramanujan-Guinand formula}\label{rggensec}
\begin{proof}[Theorem \textup{\ref{recpairkosh}}][]
Note that from Theorems \ref{thm:Kexpansion}, \ref{kzwsmall} and from the bound \cite[Eqn. (2.11)]{koshkernel}
\begin{equation*}
\left\lvert{\cos \left( {\pi z} \right){M_{2z}}(4\sqrt {tx} ) - \sin \left( {\pi z} \right){J_{2z}}(4\sqrt {tx} )}\right\rvert \ll_z 
			\begin{cases}
				1+|\log (tx)|, &\mbox{ if } \quad z=0, 0 \leq tx \leq 1, \\
				(tx)^{-|\textup{Re}(z)|}, &\mbox{ if } \quad z\neq 0,  0 \leq tx \leq 1, \\
				(tx)^{-1/4}, &\mbox{ if } \quad tx \geq 1,
				\end{cases}	
\end{equation*}
we see that the integrals in Theorem \ref{recpairkosh} indeed converge for $-\frac{1}{2}<$ \textup{Re}$(z)<\frac{1}{2}$. From \cite[Lemma 5.1]{dixitmoll}, for $\pm$Re$(z)<$ Re$(s)<3/4$ and $x>0$, we have
\begin{align}\label{1stmel}
\int_{0}^{\infty}t^{s-1}( \cos(\pi z) M_{2z}(4 \sqrt{xt}) -
\sin(\pi z) J_{2z}(4 \sqrt{xt}))dt
=\frac{\G(s-z)\G(s+z)}{\pi 2^{2s}x^s}\left(\cos(\pi z)+\cos(\pi s)\right).
\end{align}
Also from \eqref{kzw}, for Re$(s)>\pm$Re$(z)$,
\begin{align}\label{2ndmel}
2\int_{0}^{\infty}t^{s-1}\b K_{z,w}(2\b t)dt=\frac{\b^{1-s}}{2}\G\bigg(\frac{s-z}{2}\bigg)\G\bigg(\frac{s+z}{2}\bigg){}_1F_{1}\bigg(\frac{s-z}{2};\frac{1}{2};-\frac{w^2}{4}\bigg){}_1F_{1}\bigg(\frac{s+z}{2};\frac{1}{2};-\frac{w^2}{4}\bigg).
\end{align}
Note that by the hypothesis, we have $-\frac{1}{2}<$ Re$(z)<\frac{1}{2}$ so that $\pm$Re$(z)<1\pm$Re$(z)$. Then by Parseval's identity \eqref{par}, \eqref{1stmel} and \eqref{2ndmel}, for $\pm$Re$(z)<c=$ Re$(s)<\min\left(\frac{3}{4},1\pm\text{Re}(z)\right)$,
\begin{align}\label{auxid0}
&2\int_{0}^{\infty} \b \,K_{z,w}(2 \b t) \left( \cos(\pi z) M_{2z}(4 \sqrt{xt}) -
\sin(\pi z) J_{2z}(4 \sqrt{xt}) \right)\, dt\nonumber\\
&=\frac{1}{2\pi i}\int_{(c)}\frac{\b^s}{2}\G\left(\frac{1-s-z}{2}\right)\G\left(\frac{1-s+z}{2}\right){}_1F_{1}\bigg(\frac{1-s-z}{2};\frac{1}{2};-\frac{w^2}{4}\bigg){}_1F_{1}\bigg(\frac{1-s+z}{2};\frac{1}{2};-\frac{w^2}{4}\bigg)\nonumber\\
&\quad\quad\quad\quad\times\frac{\G(s-z)\G(s+z)}{\pi 2^{2s}x^s}\left(\cos(\pi z)+\cos(\pi s)\right)\, ds\nonumber\\
&=\frac{e^{-w^2/2}}{2\pi i}\int_{(c)}\frac{\b^s}{2}\G\left(\frac{1-s-z}{2}\right)\G\left(\frac{1-s+z}{2}\right){}_1F_{1}\bigg(\frac{s+z}{2};\frac{1}{2};\frac{w^2}{4}\bigg){}_1F_{1}\bigg(\frac{s-z}{2};\frac{1}{2};\frac{w^2}{4}\bigg)\nonumber\\
&\quad\quad\quad\quad\quad\times\frac{\G(s-z)\G(s+z)}{\pi 2^{2s}x^s}\left(\cos(\pi z)+\cos(\pi s)\right)\, ds,
\end{align}
by an application of \eqref{kft}. But
\begin{align}\label{auxid1}
\G\left(\tfrac{1-s-z}{2}\right)\G\left(\tfrac{1-s+z}{2}\right)\G(s-z)\G(s+z)\left(\cos(\pi z)+\cos(\pi s)\right)=\pi 2^{2s-1}\G\left(\tfrac{s-z}{2}\right)\G\left(\tfrac{s+z}{2}\right).
\end{align}
To see this, apply the duplication formula \eqref{dup} to represent each of the gamma functions $\G(s-z)$ and $\G(s+z)$ on the left side in terms of two gamma functions, and then write the two cosines in terms of gamma functions using \eqref{ref2} and simplify.
Thus from \eqref{auxid0} and \eqref{auxid1} and from the fact that $\b=1/\a$,
\begin{align*}
&2\int_{0}^{\infty} \b \,K_{z,w}(2 \b t) \left( \cos(\pi z) M_{2z}(4 \sqrt{xt}) -
\sin(\pi z) J_{2z}(4 \sqrt{xt}) \right)\, dt\nonumber\\
&=\frac{e^{-w^2/2}}{2\pi i}\int_{(c)}\G\left(\frac{s-z}{2}\right)\G\left(\frac{s+z}{2}\right){}_1F_{1}\bigg(\frac{s+z}{2};\frac{1}{2};\frac{w^2}{4}\bigg){}_1F_{1}\bigg(\frac{s-z}{2};\frac{1}{2};\frac{w^2}{4}\bigg)2^{s-2}(2\a x)^{-s}\, ds\nonumber\\
&=e^{-w^2/2}K_{z,iw}(2\a x),
\end{align*}
where the last step utilizes the definition \eqref{kzw} of $K_{z,w}(x)$. This proves the first result in \eqref{recpairkosheqn}. Since the second one can be proved in a similar way, we refrain from giving the proof. 
\end{proof}
We are now ready to prove Theorem \ref{genrgeq}. 
\begin{proof}[Theorem \textup{\ref{genrgeq}}][]
We first prove the result for a fixed $z$ such that $-1<$ Re$(z)<1$ and later extend it by analytic continuation. We begin with a result of Guinand \cite[equation (1)]{guinand}, namely, if $f(x)$ and  $f'(x)$ are integrals, $f$ tends to zero as $x\to\infty$, $f(x), xf'(x)$, and $x^2f''(x)$ belong to $L^{2}(0,\infty)$, and
\begin{equation}\label{recpairgui}
g(x) = 2\pi \int_{0}^{\infty} f(t) \left(\cos\left(\frac{\pi z}{2}\right) M_{z}(4\pi\sqrt{xt}) - \sin\left(\frac{\pi z}{2}\right) J_{z}(4\pi\sqrt{xt}) \right) \,dt,
\end{equation}
then the following transformation holds:
\begin{align}\label{guitra}
&\sum_{n=1}^{\infty} \sigma_{-z} (n) n^{\frac{z}{2}} f(n) - \zeta(1+z) \int_{0}^{\infty} x^{\frac{z}{2}} f(x) \,dx - \zeta(1-z) \int_{0}^{\infty} x^{-\frac{z}{2}} f(x) \,dx\nonumber\\
&= \sum_{n=1}^{\infty} \sigma_{-z} (n) n^{\frac{z}{2}} g(n) - \zeta(1+z) \int_{0}^{\infty} x^{\frac{z}{2}} g(x) \,dx - \zeta(1-z) \int_{0}^{\infty} x^{-\frac{z}{2}} g(x) \,dx.
\end{align}
The restrictions on $f$ given above are needed so as to have $f$ and $g$ to be reciprocal functions in the Koshliakov kernel $\left(\cos\left(\frac{\pi z}{2}\right) M_{z}(4\pi\sqrt{xt}) - \sin\left(\frac{\pi z}{2}\right) J_{z}(4\pi\sqrt{xt}) \right)$. (See \cite[Lemma $\beta$]{guian} for an analogous result.) For example, if we let $f(x)=K_{\frac{z}{2}}(2\pi\a x)$ and let $\a\b=1$, the above result gives a very short proof of \eqref{mainagain}. See \cite[Section 7]{dixitmoll}. The proof of \eqref{mainagain} by Guinand himself in \cite{guinand}, where \eqref{guitra} is given, is longer. This suggests that Guinand was unaware of the existence of Koshliakov's result \eqref{koshlyakov-1}.

Let $f(x)=e^{-\frac{w^2}{2}}K_{\frac{z}{2},iw}(2\pi\a x)$ in \eqref{guitra}. Then replacing $z$ by $z/2$, $x$ by $\pi x$ and $t$ by $\pi t$ in the second identity in \eqref{recpairkosheqn} and comparing with \eqref{recpairgui}, we find that when $-1<$ Re$(z)<1$, $g(x)=\b K_{\frac{z}{2},w}(2\pi\b x)$. Now \eqref{kzw} implies that for Re$(s)>\pm$Re$(\frac{z}{2})$,
\begin{align}\label{e0}
&\int_{0}^{\infty} x^{s-1} f(x) \,dx = \frac{e^{-\frac{w^2}{2}}}{4(\pi\a)^{s}} 
\Gamma\left(\frac{s}{2} - \frac{z}{4}\right) \Gamma\left(\frac{s}{2}+ \frac{z}{4}\right)\,_1F_1\bigg(\frac{s}{2} - \frac{z}{4};\frac{1}{2};\frac{w^2}{4}\bigg)  
\,_1F_1\bigg(\frac{s}{2} + \frac{z}{4};\frac{1}{2};\frac{w^2}{4}\bigg).
\end{align}
Since Re$(z)>-1$, we can let $s=1+\frac{z}{2}$ in the above equation so that
\begin{align}\label{e1}
\int_{0}^{\infty} x^{\frac{z}{2}} f(x) \,dx 
&= \frac{1}{4} \pi^{-\frac{(1+z)}{2}} \a^{-1-\frac{z}{2}} e^{-\frac{w^2}{4}} 
\Gamma\left(\frac{1+z}{2}\right) 
\,_1F_1\left(\frac{1+z}{2};\frac{1}{2};\frac{w^2}{4}\right),
\end{align}
since ${}_1F_{1}\left(\frac{1}{2};\frac{1}{2};\frac{w^2}{4}\right)=e^{\frac{w^2}{4}}$. Since Re$(z)<1$, we can let $s=1-\frac{z}{2}$ in \eqref{e0} whence
\begin{align}\label{e2}
\int_{0}^{\infty} x^{-\frac{z}{2}} f(x) \,dx 
&= \frac{1}{4} \pi^{-\frac{(1-z)}{2}} \a^{-1+\frac{z}{2}} e^{-\frac{w^2}{4}} 
\Gamma\left(\frac{1-z}{2}\right) 
\,_1F_1\left(\frac{1-z}{2};\frac{1}{2};\frac{w^2}{4}\right).
\end{align}
Furthermore, for Re$(s)>\pm$Re$(\frac{z}{2})$,
\begin{align*}
&\int_{0}^{\infty} x^{s-1} g(x) \,dx = \frac{\b^{1-s}}{4\pi^{s}}  
\Gamma\left(\frac{s}{2} - \frac{z}{4}\right) \Gamma\left(\frac{s}{2} + \frac{z}{4}\right)\,_1F_1\bigg(\frac{s}{2} - \frac{z}{4};\frac{1}{2};-\frac{w^2}{4}\bigg)  
\,_1F_1\bigg(\frac{s}{2} + \frac{z}{4};\frac{1}{2};-\frac{w^2}{4}\bigg),
\end{align*}
so that
\begin{align}\label{e3}
 \int_{0}^{\infty} x^{\pm \frac{z}{2}} g(x) \,dx 
&= \frac{1}{4} \pi^{-\frac{1}{2}\mp\frac{z}{2}} \b^{\mp\frac{z}{2}} e^{-\frac{w^2}{4}} 
\Gamma\left(\frac{1\pm z}{2}\right) 
\,_1F_1\left(\frac{1\pm z}{2};\frac{1}{2};-\frac{w^2}{4}\right).
\end{align}
Hence from \eqref{guitra}, \eqref{e1}, \eqref{e2} and \eqref{e3}, we see that
\begin{align*}
&\sum_{n=1}^{\infty} \sigma_{-z} (n) n^{\frac{z}{2}} e^{-\frac{w^2}{2}} K_{\frac{z}{2},iw}(2 \pi \a n)- \frac{1}{4}\pi^{-\frac{(1+z)}{2}} \a^{-1-\frac{z}{2}} e^{-\frac{w^2}{4}} \Gamma\bigg(\frac{1+z}{2}\bigg) \zeta(1+z)
\,_1F_1\bigg(\frac{1+z}{2};\frac{1}{2};\frac{w^2}{4}\bigg)\nonumber\\
&\quad-\frac{1}{4}\pi^{-\frac{(1-z)}{2}} \a^{-1+\frac{z}{2}} e^{-\frac{w^2}{4}} \Gamma\bigg(\frac{1-z}{2}\bigg) \zeta(1-z) 
\,_1F_1\bigg(\frac{1-z}{2};\frac{1}{2};\frac{w^2}{4}\bigg)\nonumber\\
&=\sum_{n=1}^{\infty} \sigma_{-z} (n) n^{\frac{z}{2}}  \b K_{\frac{z}{2},w}(2 \pi \b n)-\frac{1}{4}\pi^{-\frac{(1+z)}{2}}\b^{-\frac{z}{2}} e^{-\frac{w^2}{4}} \Gamma\bigg(\frac{1+z}{2}\bigg)\zeta(1+z)  
\,_1F_1\bigg(\frac{1+z}{2};\frac{1}{2};-\frac{w^2}{4}\bigg)\nonumber\\
&\quad-\frac{1}{4}\pi^{-\frac{(1-z)}{2}} \b^{\frac{z}{2}} e^{-\frac{w^2}{4}} \Gamma\bigg(\frac{1-z}{2}\bigg)\zeta(1-z)
\,_1F_1\bigg(\frac{1-z}{2};\frac{1}{2};-\frac{w^2}{4}\bigg).
\end{align*}
Now multiply both sides of the above equation by $4\sqrt{\a} e^{w^2/4}$, apply the relation $\a\b=1$, use the functional equation of the Riemann zeta function, namely \eqref{fe}, and simplify to arrive at \eqref{genrgeqeqn}. This completes the proof of Theorem \ref{genrgeq} for $-1<$ Re$(z)<1$. Note that both sides are analytic, as functions of $z$, in $\mathbb{C}\backslash\{-1,1\}$ since the poles of $\G\left(\pm\frac{z}{2}\right)$ at $z=\mp 2, \mp 4, \cdots$ are the trivial zeros of $\zeta(\pm z)$. Hence the result holds in $\mathbb{C}\backslash\{-1,1\}$ by analytic continuation. 
\end{proof}
\vspace{2mm}
\textbf{Remark 4}: Note that while Theorem \ref{genrgeq} holds for any $z\in\mathbb{C}\backslash\{-1,1\}$, Theorem \ref{genrg} holds for \emph{any} $z\in\mathbb{C}$. This is easily seen from the fact that the two expressions in curly brackets on the right side of \eqref{genrgeqn} vanish at $z=1$ and $-1$ respectively.
\begin{proof}[Corollary \textup{\ref{genrgeq0}}][]
The Laurent series expansion of the gamma function is given by \cite[p.~903, formula \textbf{8.321}, no. 1]{grn}
\begin{equation}\label{gammex}
\Gamma(z)=\df{1}{z}-\gamma +\cdots,
\end{equation}
where as the power series expansion of $\zeta(z)$ around $z=0$ is given by \cite[p.~19-20, Equations (2.4.3), (2.4.5)]{titch}
\begin{equation}\label{zetex}
\zeta(z) =-\df{1}{2}-\df{1}{2}\log(2\pi)z + \cdots,
\end{equation}
Also,
\begin{align}\label{powerex}
\left(\frac{\a}{\pi}\right)^{\frac{z}{2}} = 1 + \frac{z}{2} \log\left(\frac{\a}{\pi}\right)+\cdots,
\end{align}
and
\begin{align}\label{1f1e}
\,_1F_1\left(\frac{1-z}{2};\frac{1}{2};\frac{w^2}{4}\right) = 1 +\frac{w^2}{4} - \frac{w^2}{4}z + \cdots.
\end{align}
From \eqref{gammex}, \eqref{zetex}, \eqref{powerex} and \eqref{1f1e},
\allowdisplaybreaks{\begin{align}\label{p1}
&\Gamma\left(\frac{z}{2}\right) \zeta(z)  \pi^{-\frac{z}{2}} \a^{\frac{z}{2}-1} 
\,_1F_1\left(\frac{1-z}{2};\frac{1}{2};\frac{w^2}{4}\right)\nonumber\\
&=\frac{1}{\a}\bigg(\frac{2}{z}-\gamma+\cdots\bigg)\bigg(-\df{1}{2}-\df{1}{2}\log(2\pi)z + \cdots\bigg)\bigg(1 + \frac{z}{2} \log\left(\frac{\a}{\pi}\right)+\cdots\bigg)\bigg(1 +\frac{w^2}{4} - \frac{w^2}{4}z + \cdots\bigg)\nonumber\\
&=\frac{1}{\a}\left\{-\frac{1}{z}\left(1 +\frac{w^2}{4}\right)+\left(1 +\frac{w^2}{4}\right)\left(\frac{\gamma}{2}-\log(2\sqrt{\pi\a})\right)+\frac{w^2}{4}
+\text{terms with positive powers of}\hspace{1mm} z\right\}.
\end{align}}
Similarly,
\begin{align}\label{p2}
&\Gamma\left(-\frac{z}{2}\right) \zeta(-z)  \pi^{\frac{z}{2}} \a^{-\frac{z}{2}-1} 
\,_1F_1\left(\frac{1+z}{2};\frac{1}{2};\frac{w^2}{4}\right)\nonumber\\
&=\frac{1}{\a}\left\{\frac{1}{z}\bigg(1 +\frac{w^2}{4}\bigg)+\bigg(1 +\frac{w^2}{4}\bigg)\bigg(\frac{\gamma}{2}-\log(2\sqrt{\pi\a})\bigg)+\frac{w^2}{4}
+\text{terms with positive powers of}\hspace{1mm} z\right\}.
\end{align}
From \eqref{p1} and \eqref{p2},
\begin{align}\label{p3}
&\lim_{z\to 0}\left\{\Gamma\bigg(\frac{z}{2}\bigg) \zeta(z)  \pi^{-\frac{z}{2}} \a^{\frac{z}{2}-1} 
\,_1F_1\bigg(\frac{1-z}{2};\frac{1}{2};\frac{w^2}{4}\bigg)+\Gamma\left(-\dfrac{z}{2}\right) \zeta(-z)  \pi^{\frac{z}{2}} \a^{-\frac{z}{2}-1} 
\,_1F_1\bigg(\frac{1+z}{2};\frac{1}{2};\frac{w^2}{4}\bigg)\right\}\nonumber\\
&=\frac{1}{\a}\left\{\left(1 +\frac{w^2}{4}\right)\left(\gamma-\log(4\pi\a)\right)+\frac{w^2}{2}\right\}.
\end{align}
Now let $z\to 0$ in \eqref{genrgeqeqn} and use \eqref{p3} as it is for the left side of \eqref{genrgeqeqn}, and again, with $\a$ replaced by $\b$ and $w$ replaced by $iw$, for the right side. This gives \eqref{genrgeqeqn0} upon simplification.
\end{proof}
\begin{proof}[Theorem \textup{\ref{xiintgenrgthm}}][]
We show that Theorem \ref{xiintgenrgthm} follows from \eqref{ramguigeneid} upon choosing the pair $(\varphi, \psi)$ of functions reciprocal in the Koshliakov kernel to be $(e^{-\frac{w^2}{2}} K_{z,iw}(2\a x), \b \,K_{z,w}(2 \b x))$, where $\a\b=1$. The reciprocal property for this choice of the pair follows from Theorem \ref{recpairkosh}. First we show that these two functions are in the diamond class $\Diamond_{\eta,\omega}$ defined in the introduction. It suffices to show only $\,K_{z,w}(x)$ as a member of the class. To that end, note that Theorem 2.3 from \cite[p.~30-31]{temme} implies that $K_{z,w}(x)$ defined by the integral in Theorem \ref{kzwint1} is analytic in $x$ in $|\arg x|<\frac{\pi}{4}$, so the $\omega$ in the definition of $\Diamond_{\eta,\omega}$ can be taken to be $\pi/4$. Now Theorem \ref{kzwsmall} implies that the first bound in \eqref{growth} is satisfied where as Theorem \ref{thm:Kexpansion} implies that $K_{z,w}(x)$ satisfies the second bound as well. This prove that $K_{z,w}(x)\in\Diamond_{\eta,\omega}$.\\

Note that from \eqref{kzw}, \eqref{kft} and \eqref{recip1},
\begin{align*}
Z_{1}(s, z, w)&=\frac{\a^{-s}}{4} \,_1F_1\left(\frac{1-s-z}{2};\frac{1}{2};-\frac{w^2}{4}\right)  \,_1F_1\left(\frac{1-s+z}{2};\frac{1}{2};-\frac{w^2}{4}\right),\nonumber\\
Z_{2}(s, z, w)&=\frac{\b^{1-s}}{4} \,_1F_1\left(\frac{s-z}{2};\frac{1}{2};-\frac{w^2}{4}\right)  \,_1F_1\left(\frac{s+z}{2};\frac{1}{2};-\frac{w^2}{4}\right) 
\end{align*}
so that
\begin{align}\label{zszw}
Z(s,z,w) &= \frac{\a^{-s}}{4} \,_1F_1\left(\frac{1-s-z}{2};\frac{1}{2};-\frac{w^2}{4}\right)  \,_1F_1\left(\frac{1-s+z}{2};\frac{1}{2};-\frac{w^2}{4}\right) \nonumber\\
&\quad+\frac{\b^{1-s}}{4} \,_1F_1\left(\frac{s-z}{2};\frac{1}{2};-\frac{w^2}{4}\right)  
\,_1F_1\left(\frac{s+z}{2};\frac{1}{2};-\frac{w^2}{4}\right) 
\end{align}
and hence from \eqref{add}, we have
\begin{align}\label{z1it}
Z\left(\frac{1+it}{2},\frac{z}{2},w\right) &= 
 \frac{1}{4\sqrt{\a}} 
\Bigg(\a^{-\frac{it}{2}} \,_1F_1\left(\frac{1-z-it}{4};\frac{1}{2};-\frac{w^2}{4}\right)  \,_1F_1\left(\frac{1+z-it}{4};\frac{1}{2};-\frac{w^2}{4}\right) \nonumber\\
&\quad+\, \a^{\frac{it}{2}} \,_1F_1\left(\frac{1-z+it}{4};\frac{1}{2};-\frac{w^2}{4}\right)  
\,_1F_1\left(\frac{1+z+it}{4};\frac{1}{2};-\frac{w^2}{4}\right) \Bigg)\nonumber\\
&=\frac{1}{4\sqrt{\a}}\nabla_{2}\left(\a,\frac{z}{2},w,\frac{1+it}{2}\right),
\end{align}
where the last step follows from \eqref{nabrho}. Moreover, from \eqref{add},
\begin{align}\label{thet}
\Theta\left(\pi n, \frac{z}{2}, w\right) = e^{-\frac{w^2}{2}} K_{\frac{z}{2},iw}(2 n \pi \a ) +  \b K_{\frac{z}{2},w}(2 n \pi \b).
\end{align} 
We now compute $R(z, w)$. To that effect, note that from \eqref{zszw} and \eqref{kft},
\begin{align*}
Z\left(1+\frac{z}{2}, \frac{z}{2}, w\right) &= \frac{e^{-\frac{w^2}{4}}}{4} \left\{ \b^{\frac{z}{2}+1} \,_1F_1\left(\frac{1+z}{2};\frac{1}{2};\frac{w^2}{4}\right) +  \a^{\frac{z}{2}} \,_1F_1\left(\frac{1+z}{2};\frac{1}{2};-\frac{w^2}{4}\right) \right\},\nonumber\\
Z\left(1-\frac{z}{2}, \frac{z}{2}, w\right) &=\frac{e^{-\frac{w^2}{4}}}{4} \left\{ \a^{\frac{z}{2}-1} 
 \,_1F_1\left(\frac{1-z}{2};\frac{1}{2};\frac{w^2}{4}\right) +  
 \b^{\frac{z}{2}} \,_1F_1\left(\frac{1-z}{2};\frac{1}{2};-\frac{w^2}{4}\right)\right\},
\end{align*}
whence
\begin{align}\label{rzw}
R(z, w)&=\frac{e^{-\frac{w^2}{4}}}{4} 
\Bigg\{\pi^{-\frac{z}{2}} \Gamma\left(\frac{z}{2}\right) \zeta(z) \left( \a^{\frac{z}{2}-1} \,_1F_1\left(\frac{1-z}{2};\frac{1}{2};\frac{w^2}{4}\right) + 
 \b^{\frac{z}{2}} \,_1F_1\left(\frac{1-z}{2};\frac{1}{2};-\frac{w^2}{4}\right) \right)\nonumber\\ 
 &\quad+ \pi^{\frac{z}{2}} \Gamma\left(-\frac{z}{2}\right) \zeta(-z) 
 \left(\a^{-\frac{z}{2}-1} \,_1F_1\left(\frac{1+z}{2};\frac{1}{2};\frac{w^2}{4}\right) +  
 \b^{-\frac{z}{2}} \,_1F_1\left(\frac{1+z}{2};\frac{1}{2};-\frac{w^2}{4}\right) \right)
 \Bigg\}
\end{align}
Thus from \eqref{ramguigeneid}, \eqref{z1it}, \eqref{thet} and \eqref{rzw} and making use of the fact $\a\b=1$, we deduce that
\begin{align}\label{xiintgenrg1}
\frac{8}{\pi\sqrt{\a}} \int_{0}^{\infty} \Xi\bigg( \frac{t+iz}{2} \bigg) \Xi\bigg( \frac{t-iz}{2} \bigg)
\frac{\nabla_{2}\left(\a,\tfrac{z}{2},w,\tfrac{1+it}{2}\right)\, dt}{\left(t^2+(z+1)^2\right)\left(t^2+(z-1)^2\right)}=\frac{e^{-\frac{w^2}{4}}}{4\sqrt{\a}}\left(\mathfrak{F}(z, w,\a)+\mathfrak{F}(z, iw,\b)\right),
\end{align}
where 
\begin{align*}
\mathfrak{F}(z, w,\a)&=\sqrt{\a}  \bigg( 
4 \sum_{n=1}^{\infty} \sigma_{-z}(n) n^{\frac{z}{2}}  e^{-\frac{w^2}{4}} K_{\frac{z}{2},iw}(2  n \pi\a)
- \Gamma\left(\frac{z}{2}\right) \zeta(z)  \pi^{-\frac{z}{2}} \a^{\frac{z}{2}-1} 
\,_1F_1\left(\frac{1-z}{2};\frac{1}{2};\frac{w^2}{4}\right) \nonumber\\
&\quad\quad- \Gamma\bigg(-\frac{z}{2}\bigg) \zeta(-z) \pi^{\frac{z}{2}}  \a^{-\frac{z}{2}-1} \,_1F_1\bigg(\frac{1+z}{2};\frac{1}{2};\frac{w^2}{4}\bigg)  
\bigg).
\end{align*}
However, Theorem \ref{genrgeq} implies that $\mathfrak{F}(z, w,\a)=\mathfrak{F}(z, iw,\b)$, on account of which \eqref{xiintgenrg1} simplifies to \eqref{xiintgenrg}.
\end{proof}
If we let $z\to 0$ in Theorem \ref{xiintgenrgthm} and note that Corollary \ref{genrgeq0} is the special case when $z\to 0$ of Theorem \ref{genrgeq}, we readily obtain the corollary given below.
\begin{corollary}
For $\a, \b>0$ such that $\a\b=1$ and $w\in\mathbb{C}$, 
\begin{align*}
&\frac{16}{\pi} \int_{0}^{\infty} \frac{\Xi\left( \frac{t}{2} \right)^2}{\left(t^2+1\right)^2}
 \Bigg(\a^{-\frac{it}{2}} \,_1F_1^2\left(\frac{1-it}{4};\frac{1}{2};-\frac{w^2}{4}\right) + 
\a^{\frac{it}{2}} \,_1F_1^2\left(\frac{1+it}{4};\frac{1}{2};-\frac{w^2}{4}\right) \Bigg)\, dt\nonumber\\
&= \sqrt{\a}e^{-\frac{w^2}{4}} 
\Bigg( 4 \sum_{n=1}^{\infty} d(n)  e^{-\frac{w^2}{4}} K_{0,iw}(2  n \pi \a)
- \frac{\gamma - \log(4\pi \a)}{\a} \left( 1 -\frac{w^2}{4} \right)  
+ \frac{w^2}{2\a} \Bigg).
\end{align*}
\end{corollary}
When $w=0$, the above corollary gives a result obtained by Koshliakov \cite[Equation (17)]{koshxi} (see also \cite[p.~169]{ingenious}).

\section{Concluding remarks and further possible work}\label{cr}
We would like to emphasize that while the generalization of $e^{-x^2}$ that we sought in \cite{dixthet} in order to obtain formulas of the type $F(w,\a)=F(iw,\b)$, $\a\b=1$, was still a known elementary function, namely $e^{-x^2}\cos(w x)$, the generalization of $K_{z}(x)$ that plays a role similar to that of $e^{-x^2}\cos(w x)$ in obtaining transformations of the form $F(z, w,\a)=F(z, iw,\b)$, $\a\b=1$, is a \emph{new} special function, namely $K_{z,w}(x)$. The goal of this paper was two-fold - to initiate the development of the theory of this generalized modified Bessel function $K_{z,w}(x)$ as well as to study its application towards obtaining the modular-type transformations of the form $F(z,w,\a)=F(z,iw,\b)$ and to use them to evaluate integrals involving the Riemann $\Xi$-function and sums of products of confluent hypergeometric functions. While we have obtained several results on $K_{z,w}(x)$, this might be scratching the tip of an iceberg considering the vast expanse of the theory of the modified Bessel function $K_{z}(x)$. Furthermore, one may want to look what the generalized Bessel functions of the first kind associated to $K_{z,w}(x)$ turn out to be, that is, $J_{z,w}(x)$ and $Y_{z,w}(x)$, and similarly, the generalized modified Bessel function of the first kind, that is, $I_{z,w}(x)$.

Some further questions, even among the topics we have studied here, appear elusive to us. For example, in addition to the Basset-type representation for $K_{0,w}(x)$ that we obtained in Theorem \ref{bassetgen}, it is natural to seek a Basset-type representation for $K_{z,w}(x)$ valid for some region in the $z$-complex plane including $z=0$. All of our attempts to find such a representation have been unsuccessful. One such attempt relied on the application of a formula from \cite[p.~125, Equation (72)]{burchnallchaundy} in \eqref{kzw}, namely, 
\begin{align*}
{}_1F_{1}(a;c;x){}_1F_{1}(a';c;x)=\sum_{r=0}^{\infty}\frac{(a)_r(a')_r}{r!(c)_r(c)_{2r}}x^{2r}{}_1F_{1}(a+a'+2r;c+2r;x).
\end{align*}
Proceeding along the similar line as in the proof of Theorem \ref{bassetgen}, performing lots of technical calculations involving double integrals, Mellin transforms and infinite expansions such as \cite[Equation (2.35)]{glassermontaldi}
\begin{align*}
&\sum_{r=0}^{\infty}\frac{\left(\frac{w^4x^2}{16}\right)^{r}}{r!\left(\frac{1}{2}\right)_r\left(\frac{1}{2}\right)_{2r}}{}_0F_{1}\left(-;\frac{1}{2}+2r;-\frac{w^2xu}{2}\right)\nonumber\\
&={}_0F_{1}\left(-;\frac{1}{2};-\frac{w^2x}{4}\left(u+\sqrt{u^2-1}\right)\right){}_0F_{1}\left(-;\frac{1}{2};-\frac{w^2x}{4}\left(u-\sqrt{u^2-1}\right)\right)\nonumber\\
&=\cos\left(w\sqrt{x}\sqrt{u+\sqrt{u^2-1}}\right)\cos\left(w\sqrt{x}\sqrt{u-\sqrt{u^2-1}}\right)
\end{align*}
finally led us to
\begin{align*}
K_{z,w}(2x)=\frac{1}{4}\int_{0}^{\infty}e^{-x\left(t+\frac{1}{t}\right)}\left(t^z+t^{-z}\right)\cos\left(w\sqrt{xt}\right)\cos\left(w\sqrt{\frac{x}{t}}\right)\frac{(1-1/t^2)\, dt}{(t-1/t)},
\end{align*}
which, unfortunately, by a change of variable, reduces to \eqref{kzwint1eqn} again. 

With reference to our first remark in the introduction, we note that M\"{u}hlenbruch and Raji \cite{muhraji} have initiated the study of generalized Maass wave forms and have also obtained their Whittaker-Fourier expansion \cite[Equation (4.21)]{muhraji} which gives, as a special case, the well-known Fourier-Bessel expansion for Maass cusp forms of weight $0$. But these Whittaker-Fourier expansions do not appear to be connected to the series in Theorem \ref{genrg}. Thus it might be interesting to see if the latter series are connected to some other generalized Maass wave forms or other modular objects.

\begin{center}
\textbf{Acknowledgements}
\end{center}

\noindent
Atul Dixit, Aashita Kesarwani and Victor H. Moll sincerely thank Nico M. Temme for obtaining the asymptotic expansion in Theorem \ref{thm:Kexpansion} which resisted their best efforts. They also thank Arindam Roy and Rahul Kumar for useful discussions. Nico M. Temme acknowledges financial support from {\emph{Ministerio de Ciencia e Innovaci\'on}}, project MTM2015-67142-P. The first author's research is partially supported by the SERB-DST grant RES/SERB/MA/P0213/\newline
1617/0021 and sincerely thanks SERB-DST for the support.\\

\begin{center}
\textbf{Appendix: Asymptotic expansion of $K_{z,w}(2x)$ for large values of $x$}\\
\vspace{1mm}
\textbf{Nico M. Temme}
\end{center}
\vspace{2mm}
For the proof of Theorem~\ref{thm:Kexpansion} we consider \eqref{kzwint1eqn}, substitute $t=y s$ with $y=\sqrt{x}$, replace the cosines by their exponentials and obtain a sum of $4$ integrals
 \begin{equation}\label{eq:01}
K_{z,w}(2x)=\frac14\sum_{\sigma,\tau}I_{\sigma,\tau},\quad I_{\sigma,\tau}=\int_0^\infty e^{-y^2\phi_{\sigma,\tau}(s)}s^{2z-1}\,ds, 
 \end{equation}
where 
 \begin{equation*}
\sigma=\pm i , \quad \tau=\pm i,
 \end{equation*}
and
\begin{equation*}
\phi_{\sigma,\tau}(s)=s^2+\frac{1}{s^2}+\frac{w(\sigma s+\tau/s)}{y}.
 \end{equation*}
An asymptotic expansion of the integral in \eqref{eq:01} can be obtained by the saddle point method, see \cite[Chapter~4]{asytemme}.

The saddle point $s_0$ follows from a  zero of
 \begin{equation}\label{eq:04}
\phi_{\sigma,\tau}^\prime(s)=\frac{2s^4-2+sw(\sigma s^2-\tau)/y}{s^3},
 \end{equation}
and will be close to the point $s=1$, because $y$ is large. We have
 \begin{equation*}
s_0= 
\begin{cases}  1 & \text{if } \quad \sigma=  \tau=i, \\ 
{\displaystyle \sqrt{1-\frac{w^2}{16y^2}}-\frac{i\,w}{4y}}& \text{if } \quad \sigma=  -\tau=i .
\end{cases} 
 \end{equation*}

To obtain the asymptotic expansion of the integral in \eqref{eq:01} we substitute
 \begin{equation*}
\phi_{\sigma,\tau}(s)-\phi_{\sigma,\tau}(s_0)=\frac12 u^2,
 \end{equation*}
with the condition that $s=0$ corresponds to $u=-\infty$ and $s=\infty$ to $u=\infty$. This gives
 \begin{equation*}
I_{\sigma,\tau}=e^{-y^2\phi_{\sigma,\tau}(s_0)}
\int_{-\infty}^\infty e^{-\frac12y^2u^2}f(u)\,du, \quad f(u)= s^{2z-1}\frac{ds}{du}.
 \end{equation*}
The expansion 
 \begin{equation*}
 f(u)=f(0)\sum_{k=0}^\infty f_k u^k, \quad f(0)= s_0^{2z-1}\left.\frac{ds}{du}\right\vert_{u=0}, \quad f_0=1,
 \end{equation*}
gives the asymptotic expansion
 \begin{equation}\label{eq:09}
 I_{\sigma,\tau}\sim e^{-y^2\phi_{\sigma,\tau}(s_0)}f(0)\frac{\sqrt{2\pi}}{y}\sum_{k=0}^\infty \frac{a_k}{y^{2k}}, \quad a_k=2^k\left(\frac12\right)_kf_{2k}.
 \end{equation}

Further information follows from an expansion $s=s_0+d_1u+d_2u^2+\ldots$, where
 \begin{equation*}
d_1=1/\sqrt{ \phi_{\sigma,\tau}^{\prime\prime}(s_0)}= \frac12\sqrt2\,s_0^2 \sqrt{\frac{y}{s_0^4y+ws_0\tau+3y}},
 \end{equation*}

 \begin{equation*}
 d_2=-\frac{d_1^2\left(\tau-2s_0^2\sigma+s_0^4\tau\right)}{s_0\left(s_0^6\sigma-3s_0^4\tau+3s_0^2\sigma-\tau\right)},
 \end{equation*}
 \begin{equation*}
d_3=
\frac{d_1^3\left(-s_0^4\tau^2-6\tau s_0^2\sigma+2s_0^{10}\tau \sigma+2\tau^2-5s_0^8\sigma^2-s_0^8\tau^2+4\tau s_0^6\sigma+5s_0^4\sigma^2\right)}{2s_0^2\left(s_0^6\sigma-3s_0^4\tau+3s_0^2\sigma-\tau\right)^2}.
\end{equation*}

Then,
\begin{equation*}
a_1=\frac{3 d_3 s_0^2+6 d_1 z s_0 d_2-3 d_2 s_0 d_1-3 d_1^3 z+2 d_1^3 z^2+d_1^3}{s_0^2 d_1}.
 \end{equation*}

All these quantities can be expanded for large values of $y$ by using the values of $s_0$ given in \eqref{eq:04}. The same for $f(0)=s_0^{2z-1}d_1$ and $\phi_{\sigma,\tau}(s_0)$ in (\ref{eq:09}).

By adding the 4 results for the 4 combinations of $\sigma$ and $\tau$, the final expansion for $K_{z,w}(2x)$ can be obtained. 
We need to use only the cases $\sigma=\tau=i$ and $\sigma=-\tau=i$, because the sums $I_{i,i}+I_{-i,-i}$ and $I_{i,-i}+I_{-i,i}$ are even functions of $w$.
Hence, we can  take twice the even parts of $I_{i,i}$ and $I_{i,-i}$.

\paragraph{The case $\sigma=\tau=i$.} In this case the saddle point is $s_0=1$, $\phi_{\sigma,\tau}(s_0)=2+2iw/y$, and
 \begin{equation*}
 f(0)=d_1=\frac12\sqrt{2}\sqrt{y/(4y+iw)}=\frac{1}{4}\sqrt{2}\left(1-\frac{i\,w}{4y}-\frac{i\,w^3}{128y^3}+O\left(y^{-4}\right)\right),
 \end{equation*}
  \begin{equation*}
\begin{array}{ll}
{\displaystyle a_1=\frac{y\left(16y+iw-64z^2y-16iz^2w\right)}{16\left(w-4iy\right)^2}=\frac{4z^2-1}{16}+\frac{i\,w\left(7-16z^2\right)}{256y}}\ +\\[8pt]
\quad\quad{\displaystyle \frac{w^2\left(5-8z^2\right)}{512y^2}+\frac{i\,w^3\left(16z^2-13\right)}{4096y^3}+O\left(y^{-4}\right).}
\end{array}
 \end{equation*}
Hence, combining these results we obtain  the contribution of $I_{\sigma,\sigma}+ I_{-\sigma,-\sigma}$. We write  
\begin{equation}\label{eq:16}
I_{i,i}+ I_{-i,-i}=\frac{\sqrt{\pi}}{y}e^{-2y^2}\bigl(\cos(2wy)P-\sin(2wy)Q\bigr),
 \end{equation}
 where $P$ and $Q$ have the approximations
 \begin{equation}\label{eq:17}
P=1+\frac{32z^2-3w^2-8}{128 y^2}+O\left(y^{-4}\right),\quad Q=\frac{w}{8y}+O\left(y^{-3}\right).
 \end{equation}

 \paragraph{The case $\sigma=i$, $\tau=-i$.} 
We have 
\begin{equation*}
s_0=\sqrt{1-\frac{w^2}{16y^2}}-\frac{i\,w}{4y},\quad \phi_{\sigma,\tau}(s_0)=2+\frac{w^2}{4y^2},
\end{equation*}
 \begin{equation*}
\begin{array}{ll}
{\displaystyle f(0)=\frac12\sqrt{2}\,s_0^{2z-1}\frac{s_0^2}{s_0^2+1}=\frac{1}{4}\sqrt{2}\Bigl(1-\frac{i\,wz}{2y}-\frac{\left(4z^2-1\right)w^2}{32y^2}\ +}\\[8pt]
\quad\quad{\displaystyle \frac{i\,w^3\left(z^2-1\right)}{48y^3}\ +O\left(y^{-4}\right)\Bigr),}
\end{array}
 \end{equation*}
and
 \begin{equation*}
\begin{array}{ll}
{\displaystyle a_1=\frac{s_0^2\left(2z^2s_0^4-3zs_0^4+s_0^4-4s_0^2+4z^2s_0^2+1+2z^2+3z\right)}{4\left(s_0^2+1\right)^4}=}
\\[8pt]
\quad{\displaystyle\frac{4z^2-1}{32}+\frac{3i\,wz}{64y}+ \frac{w^2\left(z^2-1\right)}{128y^2}+\frac{9i\,w^3z}{2048y^3}+O\left(y^{-4}\right).}
\end{array}
\end{equation*}
Hence,
\begin{equation}\label{eq:21}
I_{i,-i}+ I_{-i,i}=\frac{\sqrt{\pi}}{y}e^{-2y^2-\frac14w^2}R, 
\end{equation}
where
\begin{equation}\label{eq:22}
R=1+\frac{(4z^2-1)(2-w^2)}{32 y^2}+O\left(y^{-4}\right).
 \end{equation}

 \paragraph{\textbf{Adding the results.}} 
Using the results in \eqref{eq:16} and \eqref{eq:21}, we find for $K_{z,w}(2x)$   (see  \eqref{eq:01})
\begin{equation*}
K_{z,w}(2x)=\frac{\sqrt{\pi}}{4y}e^{-2y^2}\bigl(\cos(2wy)P-\sin(2wy)Q+e^{-\frac14w^2}R\bigr), 
\end{equation*}
where $y=\sqrt{x}$, with first terms expansions given in \eqref{eq:17} and \eqref{eq:22},  and $x\to\infty$.

For $w=0$ this becomes
\begin{equation*}
K_{z,0}(2x)=\frac{\sqrt{\pi}}{4y}e^{-2y^2}\bigl(P+ R\bigr), 
\end{equation*}
which corresponds with the result for the modified Bessel function:
\begin{equation*}
K_{z}(2x)=\frac12\sqrt{\frac{\pi}{x}}e^{-2x}\Bigl(1+\frac{4z^2-1}{16x}+O\left(x^{-2}\right)\Bigr), \quad x\to\infty.
\end{equation*}

\end{document}